\documentclass[12pt]{article}
\def\version{3.6.2015}\def\users{}  %
\def\users{final-layout}  
\usepackage{graphicx}
\usepackage{graphics}
\usepackage{amsfonts}
\usepackage{amssymb}
\usepackage{amsmath}
\usepackage{amsthm}
\usepackage{mathrsfs}
\usepackage{color}
\usepackage{psfrag}     
\usepackage{epstopdf}
\usepackage{hyperref}

\usepackage{ifthen}
\ifthenelse{\equal{\users}{final-layout}}{}{
\usepackage{fancyhdr}
\pagestyle{fancy}
\headheight=28pt\headwidth=16.5cm
\definecolor{gray}{gray}{0.5}
\rhead{\color{green}Plasticity with damage at small strains\\
T.Roub\'\i\v cek \& J.Valdman}
\chead{}
\lhead{\color{green}Version \version, file: \jobname.tex\\
compiled:
\number\day.\number\month.\number\year\ at
\the\hour:\ifnum\minute<10 0\fi\the\minute\ h }
}

\numberwithin{equation}{section}

\newcount\hour \newcount\minute
\hour=\time
\divide \hour by 60
\minute=\time
\loop \ifnum \minute > 59 \advance \minute by -60 \repeat

\usepackage[normalem]{ulem}
\definecolor{brown}{rgb}{0.5,0,0}
\ifthenelse{\equal{\users}{final-layout}}{

    \newcommand{\DELETE}[1]{}
    
    \newcommand{\COMMENT}[1]{}
    
    \newcommand{\REM}[1]{\marginpar{\bfseries\tiny{\color{blue}}}}
    \newcommand{\TODO}[1]{} 
}
{\usepackage[notcite,notref,color]{showkeys}
\definecolor{labelkey}{rgb}{1.,.2,0.}

 \newcommand{\DELETE}[1]{{\color{brown}\sout{#1}\color{black}}}
 
 \newcommand{\COMMENT}[1]{{\color{red}\uuline{#1}\color{black}}}
 
 \newcommand{\REM}[1]{\marginpar{\bfseries\tiny{\color{blue}#1}}}
 \newcommand{\TODO}[1]{{$^{\color{blue}{TODO:}}$\footnote{\color{blue}{#1}}}}  
}

\newtheorem{theorem}{Theorem}[section]

\newtheorem{proposition}[theorem]{Proposition}

\newtheorem{definition}[theorem]{Definition}
\newtheorem{remark}[theorem]{Remark}

\newcommand\DT[1]{\mathchoice
                 {{\buildrel{\hspace*{.1em}\text{\LARGE.}}\over{#1}}}
                 {{\buildrel{\hspace*{.1em}\text{\Large.}}\over{#1}}}
                 {{\buildrel{\hspace*{.1em}\text{\large.}}\over{#1}}}
                 {{\buildrel{\hspace*{.1em}\text{\large.}}\over{#1}}}}

\newcommand{\GD}{\mathchoice
                  {\Gamma_{\hspace*{-.15em}\mbox{\tiny\rm D}}}
                  {\Gamma_{\hspace*{-.15em}\mbox{\tiny\rm D}}}
                  {\Gamma_{\hspace*{-.1em}\mbox{\tiny\rm D}}}
                  {\Gamma_{\hspace*{-.05em}\mbox{\tiny\rm D}}}}
\newcommand{\GN}{\Gamma_{\hspace*{-.15em}\mbox{\tiny\rm N}}}
\newcommand{\Sdir}{\Sigma_{\mbox{\tiny\rm D}}}
\newcommand{\Snew}{\Sigma_{\mbox{\tiny\rm N}}}
\newcommand{\sY}{\sigma_{\mbox{\tiny\rm Y}}^{}}
\newcommand{\scrE}{\mathscr E}
\newcommand{\scrR}{\mathscr R}

\newcommand\eq{\eqref}
\newcommand\ti{\times}
\newcommand\pl{\partial}
\newcommand\eps{\varepsilon}
\newcommand{\R}{{\mathbb R}} 
\newcommand{\N}{{\mathbb N}} 
\newcommand\uD{u_{\mbox{\tiny\rm D}}}
\newcommand\DTuD{\DT u_{\mbox{\tiny\rm D}}}
\newcommand\wD{w_{\mbox{\tiny\rm D}}}

\newcommand{\weak}{\rightharpoonup}
\newcommand{\weaks}{\stackrel{*}{\rightharpoonup}}
\renewcommand{\d}{\mathrm{d}}
\newcommand{\dd}{\,\mathrm{d}}
\newcommand{\Diss}{\mathrm{Diss}}
\newcommand{\frM}{f}

\newcommand{\Vdots}{\mathchoice{\,\vdots\,}
        {\,\begin{minipage}[c]{.1em}\vspace*{-.3em}$^{\vdots}$\end{minipage}\,}
        {\,\tiny\vdots\,}{\,\tiny\vdots\,}}

\newcommand{\wb}[1]{\mathchoice{\text{\large$\hspace*{.03em}\overline{\text{\normalsize$\hspace*{-.03em}#1\hspace*{-.03em}$}}\hspace*{.03em}$}}
                      {\text{\large$\hspace*{.03em}\overline{\text{\normalsize$\hspace*{-.03em}#1\hspace*{-.03em}$}}\hspace*{.03em}$}}
                     {\text{\normalsize$\hspace*{.0em}\overline{\text{\normalsize$#1$}}\hspace*{.0em}$}}
                      {\text{\small$\bar{\text{\tiny$#1$}}$}}}

\DeclareMathOperator*{\argmin}{argmin}
\DeclareMathOperator*{\tr}{tr}
\newcommand\uDtaukk{u_{\mbox{\tiny\rm D},\tau}^{k-1}}
\newcommand\uDtau{\baru_{\mbox{\tiny\rm D},\tau}}

\newcommand\uDtauhk{u_{\mbox{\tiny\rm D},\tau h}^k}
\newcommand\DEV{\R_\mathrm{dev}^{d\ti d}}
\newcommand\SYM{\R_\mathrm{sym}^{d\ti d}}
\newcommand{\baru}{\wb{u}}
\newcommand{\barz}{\wb{z}}
\newcommand{\bare}{\wb{e}}
\newcommand{\barf}{\wb{f}}
\newcommand{\barg}{\wb{g}}
\newcommand{\barpi}{\wb{\pi}}
\newcommand{\barxi}{\wb{\xi}}
\newcommand{\barzeta}{\wb{\zeta}}
\newcommand{\bareta}{\wb{\eta}}
\newcommand{\barsigma}{\wb{\sigma}}
\newcommand{\ul}{\underline}

\DeclareMathOperator*{\dev}{dev}

\newcommand{\TIME}[1]{$^{^{^{^{^{^{^\text{\small$t=#1$}}}}}}}$}

\newcommand{\pseudo}[1]{(pseudo)\discretionary{-}{}{)}#1}

\begin{document}

\vspace*{1em}


\noindent{\LARGE\bf
Stress-driven solution to rate-independent\\[.2em]
elasto-plasticity with damage 
at small strains\\[.2em]and its computer implementation.}

\bigskip\bigskip

\noindent{\large\sc Tom\'{a}\v{s} Roub\'\i\v{c}ek}\\
{\it Mathematical Institute, Charles University, Sokolovsk\'a 83,
CZ--186~75~Praha~8
} and 
{\it Institute of Thermomechanics, Czech Acad.\ Sci., Dolej\v skova~5,
CZ--182~08 Praha 8
} and
{\it Institute of Information Theory and Autom., Czech Acad.\ Sci.,
CZ--182~08 Praha~8, Czech Republic.
}

\bigskip

\noindent{\large\sc Jan Valdman}\\
{\it Institute of Information Theory and Automation, Czech Acad. Sci., Pod vod\'{a}renskou v\v{e}\v{z}\'{\i} 4, CZ--182~08 Praha~8}
and
{\it Institute of Mathematics and Biomathematics, Faculty of Science, 
University of South Bohemia, Brani\v sovsk\' a~31, 
CZ--370 05 \v{C}esk\'{e} Bud\v{e}jovice, Czech Republic.}

\bigskip\medskip

\begin{center}\begin{minipage}[t]{16.5cm}

{\small \noindent {\it Abstract.} \baselineskip=11pt
The quasistatic rate-independent 
damage combined with linearized plasticity with hardening 
at small strains is investigated. The fractional-step time discretisation
is devised with the purpose to obtain a numerically efficient
scheme converging possibly to a physically relevant stress-driven 
solutions, which however is to be verified a-posteriori by
using a suitable integrated variant of the maximum-dissipation principle.
Gradient theories both for damage
and for plasticity are considered to make the scheme numerically
stable with guaranteed convergence within the class of weak solutions.
After finite-element approximation, 
this scheme is computationally implemented and illustrative 
2-dimensional simulations are performed.
\medskip


\medskip

}
\end{minipage}
\end{center}

\vspace{1.5em}

\section{\large INTRODUCTION}

A combination of plasticity and damage, also called \emph{ductile damage},
open colorful scenarios with important applications in civil or 
mechanical engineering and with interesting mathematical problems,
in particular in comparison with mere plasticity or mere damage.
Often, both plastification and damage processes are much faster than
the rate of applied load and, in a basic scenario, any internal 
time scale is neglected and the mentioned inelastic processes
are considered as {\it rate independent}. The goal of this article 
is to devise a model together with its efficient computational 
approximation that would lead to a numerical stable and convergent
scheme and, at least in particular situations, calculate a physically
relevant solutions of a stress-driven type verifiable aposteriori 
by checking a suitable version of the maximum-dissipation principle.  

We use the very standard {\it linearized, associative,
plasticity} at small strain as presented e.g.\ in \cite{HanRed99PMTN}.
Simultaneously, we use also a rather standard {\it scalar (i.e.\ isotropic)
damage} as presented e.g.\ in \cite{Frem02NST}. 
We have 
primarily in mind a conventional engineering model with {\it unidirectional}
evolution of damage; actually, the healing will be here allowed rather 
from analytical reasons and can be expected ineffective in usual 
applications, cf.\ Remark~\ref{rem-no-healing} below.
All rate-dependent phenomena (as inertia or heat conduction and 
thermo-coupling) are neglected, this means the problem is considered 
as {\it quasistatic} and fully rate-independent. To avoid serious 
mathematical and computation difficulties, we have in mind 
an incomplete damage. 


The mentioned modelling simplification leading to quasistatic rate-independent
system however, which reflects a certain well-motivated asymptotics, brings 
however quite serious questions and difficulties because the class of 
reasonably general solutions is very wide if the governing energy is not
convex (as necessarily here) and involves solutions of very different 
nature, some of them physically not relevant; cf.\ 
\cite{Miel11DEMF,MieRou15RIST}. In particular, to avoid unwanted 
effects of unphysically easy damage under subcritical stress,
one cannot require (otherwise attractive idea of) energy conservation 
and thus cannot consider so-called energetic solutions in the
sense \cite{MieThe04RIHM}, although this concept is occasionally
used for damage with plasticity in purely mathematically-focused 
literature \cite{AlMaVi14GDMC,AlMaVi14GDMC+,Cris??GSQE}. This is related
with the discussion whether rather energy or rather stress is 
responsible for governing evolution of rate-independent systems 
\cite{Legu02STCC}. 

In contrast to the mentioned energetic solutions (which allows for simpler 
analysis without considering gradient plasticity but lead to recursive 
global-minimization problems which are difficult to realize and may slide to
unphysically scenarios of unrealistic early damage), we will 
focus here on solutions that are rather stress driven and that can be 
efficiently obtained numerically. We will rely on a certain careful
usage of a suitable integral-version of the maximum-dissipation principle, 
as devised in \cite{Roub15MDLS} and used, rather heuristically, in
engineering models of damage with plasticity and hardening, cf.\ 
\cite{ConCuo??NTCC}. This brings specific difficulties  
with convergence (which requires usage of gradient plasticity) and 
specific a-posteriori verification of a suitable 
approximate version of the mentioned maximum-dissipation principle,
as suggested in \cite[Rem.\,4.6]{Roub15MDLS} for damage itself, modified 
here for the combination of damage and plasticity analogously like
in \cite{RoPaMa??LSAR} for a surface variant of the elasto-plasto-damage
model. If the maximum-dissipation principle holds (at least with
a good accuracy) we can claim that the numerically obtained 
solution is physically relevant as stress-driven (with a good accuracy).

A physically more justified and better motivated approach would be
to involve a small viscosity to the damage variable or to 
the elastic and the plastic strains, and then to pass these
viscosities to zero. The limits obtained by this way are called 
vanishing-viscosity solutions to the original rate-independent 
system and their analysis and computer implementation is 
very difficult; for models without plasticity cf.\ \cite{KnRoZa13VVAR} 
for viscosity in damage or \cite{RoPaMa13QACV} for viscosity
in elastic strain.

In principle, there are two basic scenarios how the material
might respond to an increasing loading: either {\it first plasticize and 
then go into damage due to hardening effects}, or first go into damage and 
then plasticize; of course, various compromising scenarios are 
possible too. The latter scenario needs a damage influenced 
yield stress and allow for no hardening (and in particular
perfect plasticity), cf.\ \cite{RouVal??PPDH}. 
Let us only remark that a damage-dependence of the yield stress 
in the fully rate-independent setting would make
the dissipation state-dependent, which brings serious difficulties as
seen e.g.\ in \cite[Sect.\,3.2]{MieRou15RIST} and, for the particular 
elasto-plasto-damage model, in \cite{AlMaVi14GDMC,AlMaVi14GDMC+,Cris??GSQE}.
In this paper, we will however concern exclusively on the former scenario, i.e.\ 
in particular, the {\it damage does not influence the yield stress}. 
Moreover, we will consider only kinematic hardening, although all the 
considerations could easily be augmented by isotropic hardening, too.
(Another essential difference from \cite{RouVal??PPDH} is that, as already
explained, the energy is intentionally not conserved in here.)

The plan of the paper is as follows: In Section~\ref{sec-model} 
we devise the model in its classical formulation and then,
in Section~\ref{sect-LS}, its suitable weak formulation
with discussing stress-driven solutions and the role of the 
maximum-dissipation principle. In Section~\ref{sec-disc},
we propose a constructive time discretisation method and 
prove its numerical stability (i.e.\ a-priori estimates) and
convergence towards weak solutions. After a further finite-element 
discretisation outlined in Section~\ref{sect-BEM}, this allows 
for efficient computer implementation of the model, which is demonstrated 
on illustrative 2-dimensional examples in Section~\ref{sect-simul}.

\section{\large THE MODEL AND ITS WEAK FORMULATION}\label{sec-model}

Hereafter, we suppose that the damageable elasto-plastic
body occupies a bounded Lipschitz domain $\Omega\subset\R^d$, $d=2$ or $3$.
We  denote by $\vec{n}$ the outward unit normal to $\partial \Omega$.
We further suppose that the boundary of $\Omega$ splits as
\[
\partial\Omega :=\Gamma= \GD\cup \GN\,,
\]
with $\GD$ and $\GN$ open subsets in the relative topology of
$\partial\Omega$, disjoint one from each other and, up to $(d{-}1)$-dimensional
zero measure, covering $\partial\Omega$. Later, the 
Dirichlet or the Neumann boundary conditions will be prescribed on 
$\GD$ and $\GN$, respectively.
Considering $T>0$ a fixed time horizon, we set
\begin{displaymath}
I:=[0,T], \qquad
Q:=(0,T){\times}\Omega, \qquad \Sigma:=I{\times}\Gamma,
\qquad \Sdir\!:=I{\times}\GD, \qquad \Snew\!:=I{\times}\GN.
\end{displaymath}
Further, $\SYM$ and $\DEV$ will denote 
the set of symmetric or symmetric trace-free (=\,deviatoric) 
$(d{\ti}d)$-matrices, respectively. For readers' convenience, let us summarize 
the basic notation used in what follows:

\vspace{.7em}

\hspace*{-1.2em}\fbox{
\begin{minipage}[t]{0.42\linewidth}
\small

$d=2,3$ dimension of the problem,

$\SYM:=\{A\in\R^{d\times d};\ A=A^\top\}$, 

$\DEV:=\{A\in\SYM;\ \mathrm{tr}\,A=0\}$, 

$u:Q\to\R^d$ displacement,


$\pi:Q\to\DEV$ plastic strain,

$\zeta:Q\to[0,1]$ damage variable,


$a>0$ activation energy for damage,

$\frM{:}\Snew\to\R^d$  applied traction force,

$g{:}Q\to\R^d$  applied bulk force (as gravity),

$\sigma_\mathrm{el}:Q\to\SYM$ elastic stress,

\end{minipage}\ \
\begin{minipage}[t]{0.52\linewidth}\small

$e_\mathrm{el}:Q\to\SYM$ elastic strain, 

$e=e(u)=e_\mathrm{el}{+}\pi=\frac12\nabla u^\top\!+\frac12\nabla u$ total 

\hspace*{13.5em}small-strain tensor,

$\mathbb C:[0,1]\to\R^{3^4}$ elasticity tensor (dependent on $\zeta$),

$\mathbb H\in\R^{3^4}$ hardening tensor (independent of $\zeta$)


$S\subset\DEV$ the elastic domain (convex, 
int\,$S\ni0$),

$\wD:\Sdir\to\R^d$  prescribed
boundary displacement,

$\kappa_1>0$ scale coefficient of the gradient of plasticity,

$\kappa_2>0$ scale coefficient of the gradient of damage,

$b>0$ activation energy for possible healing.

\end{minipage}\medskip
}

\vspace{-.5em}

\begin{center}
{\small\sl Table\,1.\ }
{\small
Summary of the basic notation used through the paper. 
}
\end{center}

\noindent The {\it state} is formed by the triple $q:=(u,\pi,\zeta)$.
Considering still a (small but fixed) regularizing parameter $\eps>0$,
the governing equation/inclusions read as:
\begin{subequations}\label{plast-dam}
\begin{align}\label{plast-dam1}
&
\mathrm{div}\,\sigma_\mathrm{el}
+g=0\ \ \ \text{ with }\ \sigma_\mathrm{el}=\mathbb C (\zeta)e_\mathrm{el}\ \text{ and }\ 
e_\mathrm{el}=e(u){-}\pi,
&&\!\!\!\!\!\!\!\!\!\!\!\!\text{\sf(momentum equilibrium)}
\\[.2em]\label{plast-dam12}
&
\partial\delta_{S}^*(\DT{\pi})
\ni\mathrm{dev}\,\sigma_\mathrm{el}
-\mathbb H\pi
+\kappa_1\Delta\pi,
&&\text{\sf(plastic flow rule)}
\\[-.2em]\label{plast-dam13}
&\partial\delta_{[-a,b]}^*(\DT\zeta)\ni-\frac12\mathbb C'(\zeta)
e_\mathrm{el}:e_\mathrm{el}
+\kappa_2
\,\mathrm{div}\big(
|\nabla\zeta|^{r-2}\nabla\zeta\big)-N_{[0,1]}(\zeta),
&&\text{\sf(damage flow rule)}
\end{align}\end{subequations}
with $\delta_{S}$ the indicator function to $S$ and $\delta_{S}^*$
its convex conjugate and with ``$\dev$'' denoting the deviatoric part
of a tensor, i.e.\ $\dev\,A:=A-\mathrm{tr}\,A/d$. Here, $[\mathbb C (\zeta)e]_{ij}$
means $\sum_{k,l=1}^d\mathbb C_{ijkl}(\zeta)e_{kl}$.
 
We employed the regularizing term with a regularizing parameter 
$\eps>0$ with an exponent to be assumed suitably big, namely $r>d$. 
This regularization will facilitate analytical well-posedness of 
the problem and, because the gradient-damage term degenerates at 
$\nabla\zeta=0$, its influence is presumably small if $\eps$ is small and 
$\nabla\zeta$ not too large. 

Of course, \eq{plast-dam} is to be completed by appropriate boundary 
conditions, e.g.
\begin{subequations}\label{plast-dam-BC}
\begin{align}\label{plast-dam-BC1}
&&&u=\wD&&\text{on }\GD,&&&&
\\\label{plast-dam-BC+}
&&&
\sigma_\mathrm{el}{\cdot}\vec{n}
=\frM&&\text{on }\GN,
\\&&&\nabla\pi\vec{n}=0\ \ \text{ and }\ \ 
\nabla\zeta{\cdot}\vec{n}=0&&\text{on }\Gamma
\end{align}\end{subequations}
with $\vec{n}$ denoting the unit outward normal to $\Omega$. 
We will consider an initial-value problem 
for \eqref{plast-dam}--\eqref{plast-dam-BC} by asking for 
\begin{align}\label{plast-dam-IC}
u(0)=u_0,\ \ \ \ \ \pi(0)=\pi_0,\ \ \text{ and }\ \ \zeta(0)=\zeta_0.
\end{align}
In fact, as $\DT u$ does not occur in \eqref{plast-dam}, $u_0$ is 
rather formal and its only qualification is to make 
$\scrE(0,u_0,\pi_0,\zeta_0)$ finite not to degrade the energy
balance \eqref{def-ls-engr} on $[0,t]$.

After considering an extension $\uD=\uD(t)$ of 
$\wD(t)$ from \eqref{plast-dam-BC1} on the whole domain $\Omega$,
it is convenient to make a substitution of $u+\uD$ instead of
$u$ into \eqref{plast-dam}--\eqref{plast-dam-BC}, we arrive to 
the problem with time-constant (even homogeneous) Dirichlet boundary
conditions. More specifically, 
\begin{subequations}\label{subst-zero-Dirichlet}
\begin{align}&\label{subst-zero-Dirichlet1}
e_\mathrm{el}\ \,\text{ in \eqref{plast-dam12} replaces by }\
e_\mathrm{el}=e(u{+}\uD){-}\pi,\ \text{ and}
\\&\label{subst-zero-Dirichlet2}
\wD\ \text{ in \eqref{plast-dam-BC1} replaces by }0.
\end{align}\end{subequations}
Assuming $(\mathbb C (\zeta)e(\uD){\cdot}\vec{n}=0$ on $\GN$ for any 
admissible $\zeta$, this transformation will keep $f$ in 
\eqref{plast-dam-BC+} unchanged.

Actually, \eqref{plast-dam12} represents rather the thermodynamical-force
balance governing damage evolution while the corresponding flow rule is written 
rather in the (equivalent) form
\begin{align}\label{flow-pi}
\DT{\pi}
\in N_{S}\big(\mathrm{dev}
\,\sigma_\mathrm{el}-\mathbb H\pi
-\kappa_1\Delta\pi\big)\ \ \ \text{ with }\ \sigma_\mathrm{el}=\mathbb C (\zeta)e_\mathrm{el}
\end{align}
and with $N_{S}$ denoting the set-valued normal-cone mapping to the convex set $S$.
An analogous remark applies to \eqref{plast-dam13}.
The system \eq{plast-dam} with the boundary conditions \eq{plast-dam-BC}
has, in its weak formulation, the structure of an abstract 
Biot-type equation (or here rather inclusion, cf.\ also e.g.\ 
\cite{Biot65MID,MieRou15RIST,Nguy14SRSG}): 
\begin{align}\label{Biot}
\partial_{\DT q}^{}\scrR(\DT q)+\partial_q^{}\scrE(t,q)\ni0
\end{align}
with suitable time-dependent stored-energy functional $\scrE$ and the 
state-dependent \pseudo{potential} of dissipative forces $\scrR$. 
Equally, as already used in \eqref{flow-pi}, one can write \eq{Biot} as a 
generalized gradient flow 
\begin{align}
\DT q\in \partial_{\xi}\scrR^*\big(-\partial\scrE(t,q)\big)
\end{align}
where $\scrR^*$ denotes the conjugate functional. The governing functionals
corresponding to \eqref{plast-dam}--\eqref{plast-dam-BC} after the 
transformation \eqref{subst-zero-Dirichlet} are:
\begin{subequations}\label{eq5:plast-dam}\begin{align}
\nonumber
&\scrE(t,u,\pi,\zeta):=\int_\Omega\!\frac12\mathbb{C}(\zeta)
\big(e(u{+}\uD(t)){-}\pi\big):\big(e(u{+}\uD(t)){-}\pi\big)
+\frac12\mathbb{H}\pi:\pi
\\[-.4em]\label{eq5:plast-dam-E}
&\hspace{8em}+\frac{\kappa_1}2|\nabla\pi|^2\!+\frac{\kappa_2}r|\nabla\zeta|^r\!
+\delta_{[0,1]}(\zeta)-g(t){\cdot}u\dd x-\int_{\GN}\!f(t){\cdot}u\dd S,
\\
&\scrR(\DT{\pi},\DT{\zeta})\equiv\scrR_1(\DT{\pi})+\scrR_2(\DT{\zeta}):=
\int_\Omega \delta_S^*(\DT{\pi})+a\DT\zeta^-+b\DT\zeta^+
\dd x
\label{eq5:plast-dam-R}
\end{align}\end{subequations}
where $z^+:=\max(z,0)$ and $z^-:=\max(-z,0)\ge0$. Note that 
the damage does not affect the hardening, which reflects the 
idea that, on the microscopical level, damage in the material
that underwent hardening develops by evolving microcracks and 
even a completely damaged material consists of micro-pieces 
that bear the hardening energy $\frac12\mathbb{H}\pi:\pi$ stored before.
This model preserves coercivity of hardening even under a complete
damage but the analysis below admits only incomplete damage. 
If $\zeta\mapsto\mathbb C(\zeta)e{:}e$ is strictly convex for any $e\ne0$,
we speak about a cohesive damage which exhibits a certain hardening effect so 
that the needed driving force increases when damage is to be accomplished. 
We can thus model quite a realistic response to various loading experiments, 
as schematically shown on Figure~\ref{fig5:plast-dam} for the case of a 
possible complete damage (whose analysis remains open, however). Note that,
due to the ``incompressibility'' constraint $\tr\,\pi=0$, no
plastification is triggered under a pure tension or compression loading.

\begin{figure}[th]
\begin{center}
\psfrag{TENSION/COMPRESSION}{\small\bf\hspace*{-2em} TENSION/COMPRESSION}
\psfrag{SHEAR LOADING}{\small\bf SHEAR LOADING}
\psfrag{[u]t}{\footnotesize $e=e(u)$}
\psfrag{elasticity}{\footnotesize elasticity}
\psfrag{damage}{\footnotesize damage}
\psfrag{plasticity}{\footnotesize plasticity with hardening}
\psfrag{stress2}{\footnotesize $\sigma_\mathrm{el}$}
\psfrag{sd}{\footnotesize $\displaystyle{\!\sqrt{\frac{2a}{C'(1)}}C(1)}$}
\psfrag{sd2}{\footnotesize $|S|$}
\psfrag{s2}{\footnotesize $|S|/|C(1)|$}
\psfrag{e3}{\footnotesize $\displaystyle{e=\sqrt{\frac{2a}{|C'(0)|}}}$}
\psfrag{e4}{\footnotesize $\displaystyle{e=\sqrt{\frac{2a}{|C'(1)|}}}$}
\psfrag{slope}{\footnotesize \!slope\,=\,$C(1)$}
\psfrag{slope2}{\footnotesize slope\,=\,$\displaystyle{\frac{C(1)H}{C(1){+}H}}$}
\hspace*{-.5em}\vspace*{-1.0em}{\includegraphics[width=.85\textwidth]{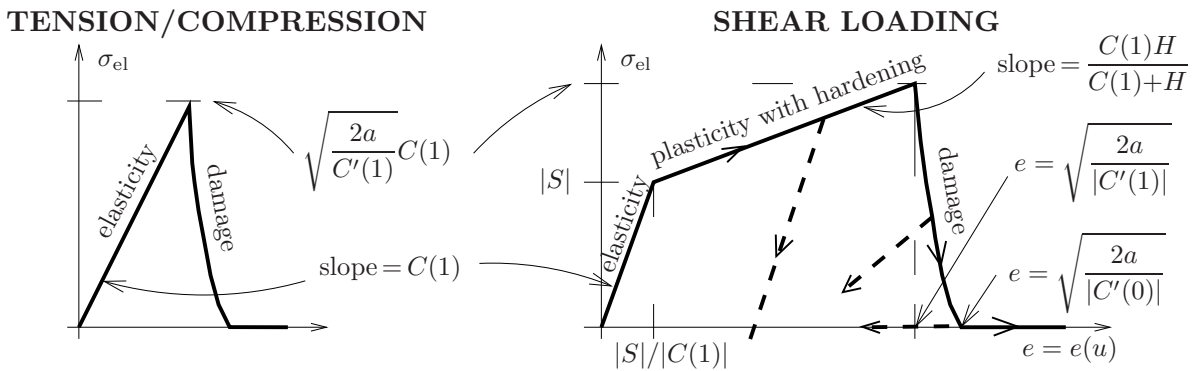}}
\end{center}\vspace*{-.5em}
\caption{\sl Schematic response of the stress $\sigma_\mathrm{el}$ 
to the total strain $e$ during a ``one-dimensional'' tension (left) or shear 
(right) loading experiment under a stress-driven scenario. The latter option 
combines plasticity with eventual complete damage. Dashed lines outline a 
response to unloading, $C=C(\zeta)$ refers to
Young's modulus (left) or the shear modulus (right).}
\label{fig5:plast-dam}
\end{figure}

Let us further note that $(u,\pi)\mapsto\scrE(t,u,\pi,\zeta)$ is smooth 
so that $\pl_q\scrE=\{\scrE_u'\}\times\{\scrE_\pi'\}\times\pl_\zeta\scrE$
with $\scrE_u'$ and $\scrE_\pi'$ denoting the respective partial
G\^ateaux derivatives and \eqref{Biot} can thus be written more specifically 
as the system:
\begin{subequations}\label{Biot-syst}\begin{align}
\scrE_u'(t,u,\pi,\zeta)&=0,\label{Biot-syst1}
\\[.1em]\partial\scrR_1(\DT\pi)+\scrE_\pi'(t,u,\pi,\zeta)&\,\ni\,0,\label{Biot-syst2}
\\[-.1em]\partial\scrR_2(\DT\zeta)+\pl_\zeta\scrE(t,u,\pi,\zeta)&\,\ni\,0.\label{Biot-syst3}
\end{align}\end{subequations}

\begin{remark}[{\sl Irreversible damage in engineering models}]\label{rem-no-healing}
\upshape
Usual engineering models consider $b=\infty$, i.e.\ no healing is allowed.
In fact, due to an essentially missing driving force for healing, 
our modification $b<\infty$ would not have any influence on
the evolution if it were not any $\nabla\zeta$-term in the stored energy.
Thus, if the healing threshold $b$ is big and the gradient-term
coefficient $\kappa_2>0$ is small, we expect to have essentially the
(usually desired) unidirectional evolution as far as the damage concerns.
\end{remark}

\begin{remark}[{\sl Surface variant of the damage/plasticity}]
\upshape
A similar scenario distinguishing tension (which leads to damage without 
plastification) and shear (with plastifying the material before damage) 
as in Figure~\ref{fig5:plast-dam} was used in a surface variant to model 
an adhesive contact distinguishing delamination in the opening and in the 
shearing modes, devised in \cite{RoKrZe13DACM,RoMaPa13QMMD} and later 
implemented by the fractional-step discretisation with checking the 
approximate maximum-dissipation principle in 
\cite{PaMaRo??TACM,RoPaMa??LSAR,VoMaRo14EMDL}. An additional analytical 
difference is that, in contrast to our bulk model here, the surface 
variant allows for irreversible damage that does not need any gradient.
\end{remark}

\begin{remark}[{\sl Other material models}]
\upshape
A separately convex stored energy $\scrE(t,\cdot)$ 
occurs also in other models. E.g., some phenomenological 
models for {\it phase transformations} in (polycrystalline) 
{\it shape-memory materials} \cite{SadBha07MICM} gives
$\zeta$ the meaning of a volume fraction (instead of damage) and 
$\pi$ a transformation strain (or a combination of the plastic and the
transformation strains), and the total strain decomposes 
as $e(u)=e_{\rm el}+\zeta\pi$ rather than \eqref{plast-dam1}
or makes $\pi$ dependent on $\zeta$ (which is then vector-valued). 
Considering the degree-1 homogeneous dissipation potential, most of 
the considerations in this paper can be applied to such a model, too; 
in fact, the only difference would be the nonsmoothness of $\scrE$ 
also with respect to $\pi$ variable. A similar (in general non-convex) model 
have been also considered in \cite{BoFrLe06GEUT,FrBeSe14MMCM,KreSte01ENFT,MieThe04RIHM,SFBBS12TMNT} 
although sometimes special choices of elastic moduli leading to convex 
were particularly under focus while the dissipation is made 
state-dependent.
\end{remark}

\section{\large LOCAL SOLUTIONS}\label{sect-LS}

We will use the standard notation $W^{1,p}(\Omega)$ for the Sobolev space
of functions having the gradient in the Lebesgue space $L^p(\Omega;\R^d)$. 
If valued in $\R^n$ with $n\ge2$, we will write $W^{1,p}(\Omega;\R^n)$, and 
furthermore we use the shorthand notation 
$H^1(\Omega;\R^n)=W^{1,2}(\Omega;\R^n)$.
We also use the notation of ``$\,\cdot\,$'' and ``$\,:\,$'' 
for a scalar product of vectors and 2nd-order tensors, respectively,
and later also ``$\,\Vdots\,$'' for 3rd-order tensors.
For a Banach space $X$, $L^p(I;X)$ will denote the Bochner space of
$X$-valued Bochner measurable functions $u:I\to X$ with its norm
$\|u(\cdot)\|$ in $L^p(I)$, here $\|\cdot\|$ stands for
the norm in $X$. Further, $W^{1,p}(I;X)$ denotes the Banach space of mappings
$u:I\to X$
whose distributional time derivative is in $L^p(I;X)$, while
${\rm BV}(I;X)$ will denote the space of mappings $u:I\to X$
with a bounded variations, i.e.\ $
\sup_{0\le t_0<t_1<...<t_{n-1}<t_n\le T}\sum_{i=1}^n\|u(t_i){-}u(t_{i-1})\|<\infty$
where the supremum is taken over all finite partitions of the interval
$I=[0,T]$. By ${\rm B}(I;X)$ we denote the space of bounded measurable
(everywhere defined) mappings $I\to X$.

The concept of local solutions has been introduced for a special crack
problem in \cite{ToaZan09AVAQ}  and independently also in
\cite{Stef09VCRI},  and further generally investigated in
\cite{Miel11DEMF}. Here, we additionally combine it with the concept of
semi-stability as invented in \cite{Roub09RIPV}.
We adapt the general definition directly to our specific problem,
which will lead to two semi-stability conditions for $\zeta$ and
$\pi$, respectively:

\begin{definition}[Local solutions]\label{def-loc-sln}
We call a measurable mapping
$(u,\pi,\zeta):I\to H^1(\Omega;\R^d)\times H^1(\Omega;\DEV)
\times W^{1,r}(\Omega)$ a local solution to the
elasto-plasto-damage problem \eqref{plast-dam}--\eqref{plast-dam-IC} if the 
initial conditions \eqref{plast-dam-IC} are satisfied and, for some 
$J\subset I$ at most countable (containing
time instances where the solution may possibly jump), it holds that:
\begin{subequations}\label{def-LS}
\begin{align}
&\!\!\!\forall t\!\in\!I{\setminus}J
:\ \ \ \scrE_u'\big(t,u(t),\pi(t),\zeta(t)\big)=0,
\label{def-ls-VI}
\\
&\!\!\!\forall t\!\in\!I{\setminus}J\ \forall\widetilde\pi\!\in\! 
H^1(\Omega;\DEV):\
\scrE\big(t,u(t),\pi(t),\zeta(t)\big)\le\scrE\big(t,u(t),\widetilde\pi,\zeta(t)\big)
+\scrR_1\big(\widetilde\pi{-}\pi(t)\big),
\label{def-ls-semi-stab1}
\\\nonumber
&\!\!\!\forall t\!\in\!I{\setminus}\{0\}\ \forall\widetilde\zeta\!\in\! W^{1,r}(\Omega),\ \,
0\!\le\!\widetilde\zeta\!\le\!1
:\ \
\\&\qquad\qquad\qquad\quad\
\label{def-ls-semi-stab2}
\scrE\big(t,u(t),\pi(t),\zeta(t)\big)\le\scrE\big(t,u(t),\pi(t),\widetilde\zeta\big)
+\scrR_2\big(\widetilde\zeta{-}\zeta(t)\big),
\\\nonumber&\!\!\!\forall0\!\le\!t_1\!\le\!t_2\!\le\!T:\ \ \,
\scrE\big(t_2,u(t_2),\pi(t_2),\zeta(t_2)\big)
+{\rm Diss}_{\scrR_1}^{}\big(\pi;[t_1,t_2]\big)
+{\rm Diss}_{\scrR_2}^{}\big(\zeta;[t_1,t_2]\big)
\\[-.2em]\label{def-ls-engr}&\hspace{11em}
\le\scrE\big(t_1,u(t_1),\pi(t_1),\zeta(t_1)\big)+\!\int_{t_1}^{t_2}\!\!
\scrE_t'(t,u(t),\pi(t),\zeta(t))\,\d t,
\end{align}\end{subequations}
where 
\begin{subequations}\label{def-LS-Diss}
\begin{align}
&\Diss_{\scrR_1}^{}(\pi;[r,s]):=
\!\!\!\!\sup_{\substack{N\in\N\\[.1em]r\leq t_0^{}<t_1^{}<\cdots<t_{N-1}^{}<t_N^{}\le s}}
\sum_{j=1}^N\int_\Omega
\delta_{S}^*(\pi(t_{j-1}){-}\pi(t_j))
\dd x,
\ \ \ \text{ and similarly }
\\&\nonumber
\Diss_{\scrR_2}^{}(\zeta;[r,s]):=
\!\!\!\!\sup_{\substack{N\in\N\\[.1em]r\leq t_0^{}<t_1^{}<\cdots<t_{N-1}^{}<t_N^{}\le s}}
\sum_{j=1}^N
\int_\Omega a(\pi(t_{j-1}){-}\pi(t_j))^-+b(\pi(t_{j-1}){-}\pi(t_j))^+\dd x.
\\[-1.7em]
\end{align}\end{subequations}
\end{definition}

Let us comment the above definition briefly. Obviously, \eqref{plast-dam1} 
after transforming the boundary condition \eq{subst-zero-Dirichlet} means 
precisely \eqref{def-ls-VI},  which more in detail here means that
$\int_{\Omega}\mathbb{C}(\zeta(t))(e(u(t){-}\wD(t))-\pi){:}e(v)\,\d x
=\int_{\Omega}g{\cdot}v\,\d x+\int_{\GN}f{\cdot}v\,\d S$
for all $v\!\in\!H^1(\Omega;\R^d)$ with $v|_{\GD}=0$, i.e.\ the weakly 
formulated Euler-Lagrange equation for displacement.
Note that \eqref{def-ls-VI} specifies also the boundary conditions
for $u$, namely $u=0$ on $\GD$ because otherwise
$\scrE(t,u,\pi,\zeta)=\infty$ would violate \eqref{def-ls-VI}
for $v$ which satisfies $v=0$ on $\GD$, and also
$\sigma_{\rm el}\cdot\vec{n}
=f$ on $\GN$ can be proved by standard arguments 
based on Green's theorem. Equivalently, one can merge \eqref{def-ls-VI}
with \eqref{def-ls-semi-stab1} to a single condition
\begin{align}\nonumber
&\!\!\!
\forall t\!\in\!I{\setminus}J\ \forall(\widetilde u,\widetilde\pi)\!\in\! 
H^1(\Omega;\R^d)\times H^1(\Omega;\DEV),\ \ \widetilde u|_{\GD}^{}=0:\quad
\\&\qquad\qquad\qquad\qquad\qquad\qquad\quad
\scrE\big(t,u(t),\pi(t),\zeta(t)\big)\le\scrE\big(t,\widetilde u,\widetilde\pi,\zeta(t)\big)
+\scrR_1\big(\widetilde\pi{-}\pi(t)\big);
\label{def-ls-semi-stab1+}
\end{align}
which reveals that Definition~\ref{def-loc-sln} just copies 
the concept of local solutions from \cite{Stef09VCRI,ToaZan09AVAQ} 
here generalized for the case of non-vanishing dissipation $\scrR_1\ne0$.
As $\scrR_1$ is homogeneous degree-1, 
always $\partial\scrR_1(\DT\pi)\subset\partial\scrR_1(0)$ and thus 
\eqref{Biot-syst2} implies
$\partial\scrR_1(0)+\partial_\pi\scrE(u,\pi,\zeta)\ni0$.
From the convexity of $\scrR_1$ when taking into account
that $\scrR_1(0)=0$, the latter inclusion is equivalent to
$\scrR_1(v)+\langle\partial_\pi\scrE(u(t),\pi(t),\zeta(t)),v\rangle\ge0$
for any $v\in H^1(\Omega;\DEV)$. Substituting $v=\widetilde z-z(t)$ and using
the convexity of
$\scrE(t,u,\zeta,\cdot)$, we obtain the \emph{semi-stability}
\eqref{def-ls-semi-stab1}  of $\pi$ at time $t$.
Analogously, we obtain also \eqref{def-ls-semi-stab2} from 
\eqref{Biot-syst3}; note that we do not require its validity at $t=0$
so that we do not need to qualify the initial conditions as far as 
any (semi)stability concerns.
Eventually, \eqref{def-ls-engr} is the (im)balance of the mechanical
energy. Note that, in view of \eq{eq5:plast-dam-E}, the last term in 
\eqref{dam-ls-semi-engr} involves
\begin{align*}
\scrE_t'(t,u,\pi,\zeta)
=\int_\Omega\!\mathbb{C}(\zeta)\big(e(\DTuD)\big):
\big(e(u{+}\uD){-}\pi\big)-\DT g(t){\cdot}u\dd x
-\int_{\GN}\!\!\DT f(t){\cdot}u\dd S.
\end{align*}
This is equivalent (or, if $\scrE(t,\cdot,\cdot,\cdot)$ is not
smooth, slightly generalizes) 
the standard definition of the {\it weak solution}
to the initial-boundary-value problem \eqref{plast-dam}--\eqref{plast-dam-IC}, 
cf.\ \cite[Prop.\,2.3]{Roub15MDLS} for details.

To be more precise, the concept of local solutions as used in
\cite{Miel11DEMF,ToaZan09AVAQ} requires $J$ only to have a zero Lebesgue
measure and also \eqref{def-ls-semi-stab2} is valid
only for a.a.\ $t$. On the other hand, conventional weak solutions
allow even \eqref{def-ls-engr} holding only for a.a.\ $t_1$ and $t_2$.
Later, our approximation method will provide convergence to this
slightly stronger local solutions, which motivates us to have tailored
Definition~\ref{def-loc-sln} straight to our results.

Actually, local solutions form essentially the largest
reasonable class of solutions for rate-independent systems as 
\eqref{plast-dam}--\eqref{plast-dam-IC} considered here. It includes
the mentioned energetic solutions \cite{Miel11DEMF,MieThe04RIHM},
the vanishing-viscosity solutions, the balanced-viscosity (so-called BV)
solutions, parametrized solutions, etc.; cf.\ \cite{Miel11DEMF,MieRou15RIST} 
for a survey, and also stress-driven-like solutions obeying 
maximum-dissipation principle in some sense, cf.\ Remark~\ref{rem-MDP}. The 
energetic solutions have often tendency
to undergo damage unphysically early; cf.\
\cite{VoMaRo14EMDL} for a comparison on several computational experiments
on a similar type of problem.
The approximation method we will use in this article leads rather to 
the stress-driven option, cf.\ Remarks~\ref{rem-MDP} and \ref{rem-DMP} below.

\begin{remark}[{\sl Maximum-dissipation principle}]\label{rem-MDP}
\upshape
The degree-1 homogeneity of $\scrR_1$ and $\scrR_2$ defined in
\eqref{eq5:plast-dam-R} allows for further interpretation of the flow
rules \eqref{Biot-syst2} and \eqref{Biot-syst3}. Using maximal-monotonicity of 
the subdifferential, \eqref{Biot-syst2} means just that
$\langle\widetilde{\xi}-\xi_\mathrm{plast},v-\DT\pi\rangle\ge0$
for any $v$ and any $\widetilde{\xi}\in\partial\scrR_1(v)$
with the 
driving force $\xi_\mathrm{plast}=-\scrE_\pi'(t,u,\pi,\zeta)$.
In particular, for $v=0$, defining the convex ``elastic domain'' 
$K_1=\partial\scrR_1(0)$, one obtains
\begin{subequations}\label{max-dis-principle}
\begin{align}\label{max-dis-principle-a}
\big\langle \xi_\mathrm{plast}(t),\DT\pi(t)\big\rangle
=\max_{\widetilde{\xi}\in K_1}\big\langle\widetilde{\xi},
\DT\pi(t)\big\rangle\ \ \ \text{ with some }\ \ \xi_\mathrm{plast}(t)=-
\scrE_\pi'(t,u(t),\pi(t),\zeta(t)).
\end{align}
To derive it, we have used that
$\xi_\mathrm{plast}\in\partial\scrR_1(\DT\pi)\subset\partial\scrR_1(0)=K_1$
thanks to the degree-0 homogeneity of
$\partial\scrR_1$, so that always $\langle \xi_\mathrm{plast},\DT\pi\rangle
\le\max_{\widetilde{\xi}\in K_1}\langle\widetilde{\xi},\DT\pi\rangle$. The 
identity \eqref{max-dis-principle-a} says that the
dissipation due to the driving force $\xi_\mathrm{plast}$
is maximal provided that the order-parameter rate $\DT\pi$ is 
kept fixed, while the vector of possible driving
forces $\widetilde{\xi}$ varies freely over all admissible driving
force from $K_1$. This just resembles the so-called Hill's
\emph{maximum-dissipation principle} articulated just for plasticity in
 \cite{Hill48VPMP}. Also it says that the rates are orthogonal
to the elastic domain $K_1$, known as an
orthogonality principle \cite{Zieg58AGOP}.
Actually, R.\,Hill \cite{Hill48VPMP} used it for a situation
where $\scrE(t,\cdot)$ is convex while, in a general nonconvex case
as also here when damage is considered, it holds only along absolutely 
continuous paths (i.e.\ in stick or slip regimes) which are sufficiently
regular in the sense $\DT\pi$ is valued not only in 
$L^1(\Omega;\DEV)$ but also in $H^1(\Omega;\DEV)^*$ 
while it does certainly not need to hold during jumps. 
Analogously it holds also for $\zeta$, defining
$K_2:=\partial\scrR_2(0)$, that
\begin{align}\label{max-dis-principle-b}
\big\langle \xi_\mathrm{dam}(t),\DT\zeta(t)\big\rangle
=\max_{\widetilde{\xi}\in K_2}\big\langle\widetilde{\xi},
\DT\zeta(t)\big\rangle\ \ \ \text{ with some }\ \ 
\xi_\mathrm{dam}(t)\in-\partial_\zeta\scrE(t,u(t),\pi(t),\zeta(t)).
\end{align}\end{subequations}
Here, $\partial_\zeta\scrE(t,u,\pi,\zeta)$ is set-valued
and its elements should be understood as  ``available'' 
driving forces not necessarily falling into $K_2$, while 
$\xi_\mathrm{dam}\in K_2$ is in a position of an ``actual'' driving force 
realized during the actual evolution. As $\scrE(t,u,\cdot,\zeta)$ is smooth,
the maximum-dissipation relation \eqref{max-dis-principle-a}
written in the form
$\langle-\scrE'_\pi(t,u(t),\pi(t),\zeta(t)),\DT\pi(t)\rangle
=\max\langle K_1, \DT \pi(t)\rangle=\scrR_1(\DT\pi(t))$
summed with the semistability \eqref{def-ls-semi-stab1}
which can be written in the form $\scrR_1(\widetilde\pi)
+\langle\scrE'_\pi(t,u(t),\pi(t),\zeta(t)),\widetilde\pi\rangle\ge0$
thanks to the convexity of $\scrE(t,u,\cdot,\zeta)$ yields
\begin{align}
\scrR_1(\widetilde\pi)
+\langle\scrE'_\pi(t,u(t),\pi(t),\zeta(t)),\widetilde\pi- \DT \pi(t) \rangle
\ge\scrR_1(\DT\pi(t))
\end{align}
for any $\widetilde\pi$, which just means that
$\xi_\mathrm{plast}(t)=-\scrE'_\pi(t,u(t),\pi(t),\zeta(t))
\in\partial\scrR_1(\DT\pi(t))$, cf.\ \eqref{Biot-syst2}. This exactly means 
that the evolution of $\pi$ is {\it governed by a thermodynamical driving force}
$\xi_\mathrm{plast}$  (we say that it is ``stress-driven'')  and it reveals
the role of the  maximum-dissipation principle in combination
with semistability. Using the convexity of $\scrE(t,u,\pi,\cdot)$,
a similar argument can be applied for
\eqref{max-dis-principle-b} in combination with semistability
\eqref{def-ls-semi-stab2} even if $\scrE(t,u,\pi,\cdot)$ is not
smooth. 
\end{remark}

\begin{remark}[{\sl Integrated maximum-dissipation principle}]
\label{rem-IMDP}
\upshape
Let us emphasize that, in general, $\DT\pi$ and $\DT\zeta$ are measures
possibly having singular parts concentrated at times when rupture occurs 
and the solution and also the driving forces need not be continuous.
Even if $\DT\pi$ and $\DT\zeta$ are absolutely continuous, in
our infinite-dimensional case the driving forces need not be in duality
 with them, as already mentioned in Remark~\ref{rem-MDP}. 
So \eqref{max-dis-principle} is analytically not justified in any sense. 
For this reason,
an {\it Integrated} version of the {\it Maximum-Dissipation
Principle} (IMDP) was devised in \cite{Roub15MDLS} for a bit simpler
case involving only one maximum-dissipation relation.
Realizing that
$\max_{\widetilde{\xi}\in K_1}\langle\widetilde{\xi},
\DT\pi\rangle=\scrR_1(\DT\pi)$ and
similarly $\max_{\widetilde{\xi}\in K_2}\langle\widetilde{\xi},
\DT\zeta\rangle=\scrR_2(\DT\zeta)$, the integrated version of
\eqref{max-dis-principle} reads here as:
\begin{subequations}\label{IMDP}\begin{align}\label{IMDP-a}
&\!\int_{t_1}^{t_2}\!\!\!\xi_\mathrm{plast}(t)\,\d\pi(t)
=\!\int_{t_1}^{t_2}\!\!\!\scrR_1(\DT\pi)\,\d t\ \ \ 
&&\text{with }\ \
\xi_\mathrm{plast}(t)=-\scrE_\pi'(t,u(t),\pi(t),\zeta(t)),\
&&\\\label{IMDP-b}
&\!\!\int_{t_1}^{t_2}\!\!\!\xi_\mathrm{dam}(t)\,\d\zeta(t)
=\!\int_{t_1}^{t_2}\!\!\!\scrR_2(\DT\zeta)\,\d t\ \ \ 
&&\text{with some }\ \:
\xi_\mathrm{dam}(t)\!\in\!-\partial_\zeta\scrE(t,u(t),\pi(t),\zeta(t))\
&&\end{align}\end{subequations}
to be valid for any $0\le t_1<t_2\le T$. This definition is inevitably
a bit technical and, without sliding too much into details, let us
only mention that the left-hand-side integrals in \eqref{IMDP}
are the lower Riemann-Stieltjes integrals suitably generalized, 
defined by limit superior of lower Darboux sums, i.e.\ 
\begin{align}\label{def-of-RS-int}
\int_r^s\!\xi(t)\,\d z(t)
:=\!\!
\limsup_{\substack{N\in\N\\[.1em]r=t_0^{}<t_1^{}<...<t_{N-1}^{}<t_N^{}=s^{^{}}}}\,
\sum_{j=1}^{N}\,\inf_{t\in[t_{j-1},t_j]}
\big\langle
\xi(t),z(t_j){-}z(t_{j-1})
\big\rangle,
\end{align}
relying on that the values of $\xi$ are in duality with values of $z$ (but
not necessarily of $\DT z$) and on that the collection of finite partitions 
of the interval $[r,s]$ forms a directed set when ordered by inclusion so
that ``limsup'' in \eqref{def-of-RS-int} is well defined. Let us mention 
that the conventional definition uses ``sup'' instead of ``limsup'' but
restricts only to scalar-valued $\zeta$ and $z$ with $z$ non-decreasing.
The limit-construction \eqref{def-of-RS-int} is called a (here \emph{lower}) 
\emph{Moore-Pollard-Stieltjes} integral \cite{Moor15DLGI,Poll23SIG}
used here for vector-valued functions in duality,
which is a very special case of a so-called multilinear Stieltjes
integral. Like in the mentioned classical scalar situation 
of lower Riemann-Stieltjes integral using ``sup'' instead of ``limsup'',
the sub-additivity of the integral with respect to $u$ and to $v$ holds, 
as well as additivity  with respect to the
domain holds.


The right-hand-side integrals in \eqref{IMDP} are just the integrals of 
measures and equal to 
$\Diss_{\scrR_1}^{}(\pi;[t_1,t_2])$ and $\Diss_{\scrR_2}^{}(\zeta;[t_1,t_2])$,
respectively. Equivalently, in view of the definition \eqref{def-LS-Diss},
they can be also written as $\int_{t_1}^{t_2}\!\scrR_1\,\d\pi(t)$ and 
$\int_{t_1}^{t_2}\!\scrR_2\,\d\zeta(t)$, where the integrals 
can again be understood as the lower Moore-Pollard-Stieltjes integrals
(or here even simpler as the mentioned lower Riemann-Stieltjes integrals)
modified for the case that the time-dependent linear functionals $\xi$
are replaced by nonlinear but time-constant and 1-homogeneous convex 
functionals $\scrR$'s.
Alternatively, though not equivalently, 
denoting the internal variables $z=(\pi,\zeta)$, the 
IMDP \eqref{IMDP} can be written 
``more compactly'' as
\begin{align}\label{IMDP+}
\!\int_{t_1}^{t_2}\!\!\xi(t)\,\d z(t)
=\!\int_{t_1}^{t_2}\!\!\scrR\,\d z(t)\ \ \ 
\ \ \ \ \text{ with some }\ \:
\xi(t)\!\in\!-\partial_z\scrE(t,u(t),z(t)).
\end{align}

Both IMDP \eqref{IMDP} or \eqref{IMDP+} are satisfied on any interval 
$[t_1,t_2]$ where the solution to \eqref{Biot-syst} is absolutely continuous 
with sufficiently regular time derivatives; then the integrals in \eqref{IMDP}
are the conventional Lebesgue integrals, in particular the left-hand
sides in \eqref{IMDP} are  
$\int_{t_1}^{t_2}\langle\xi_\mathrm{plast}(t),\DT\pi(t)\rangle\,\d t$
and $\int_{t_1}^{t_2}\langle\xi_\mathrm{dam}(t),\DT\zeta(t)\rangle\,\d t$, 
respectively. The particular importance of IMDP is  especially  at jumps, 
i.e.\ at times when abrupt damage possibly happens.
It is shown in \cite{MieRou15RIST,Roub15MDLS} on various
finite-dimensional examples of ``damageable springs'' that this
IMDP can identify too early rupturing local solutions  when the
driving force is obviously unphysically low (which occurs quite typically  
in particular  within  the energetic  solutions of systems
governed by nonconvex potentials like here) and its satisfaction for 
left-continuous local solutions indicates that the evolution is stress driven,
as explained in Remark~\ref{rem-MDP}. On the other hand, it does not need to 
be satisfied even in physically well justified stress-driven local solutions. 
For example, it happens if two springs with different fracture toughness 
organized in parallel rupture at the same time, cf.\ 
\cite[Example~4.3.40]{MieRou15RIST}, although even in this situation
our algorithm \eqref{dam-small-II-LS-min} below will give 
a correct approximate solution, cf.~Figure~\ref{fig:residua_plus_diss_energy}
below. Therefore, even the IMDP \eqref{IMDP} may serve only as a 
sufficient aposteriori condition whose satisfaction verifies 
the obtained local solution as a physically relevant in the sense
that it is stress driven but its dissatisfaction does not mean anything.
Eventually, let us realize that, as a consequence of the mentioned 
definitions, we have 
\begin{subequations}\begin{align}
\label{sub-aditivity}
&\int_{t_1}^{t_2}\!\!\!\xi_\mathrm{plast}(t)\,\d\pi(t)
+\int_{t_1}^{t_2}\!\!\!\xi_\mathrm{dam}(t)\,\d\zeta(t)\le
\int_{t_1}^{t_2}\!\!\!\xi(t)\,\d(\pi,\zeta)(t)\ \ \text{ and}
\\&\int_{t_1}^{t_2}\!\!\scrR_1\,\d\pi(t)
+\int_{t_1}^{t_2}\!\!\!\scrR_2\,\d\zeta(t)=
\int_{t_1}^{t_2}\!\!\!\big(\scrR_1{+}\scrR_2\big)\,\d(\pi,\zeta)(t).
\end{align}\end{subequations}
As there is only inequality in \eqref{sub-aditivity}, the IMDP 
\eqref{IMDP+} is less selective than \eqref{IMDP} in general. Moreover, 
we will rely rather on some approximation of IMDP, cf.\
Remarks~\ref{rem-DMP} and \ref{rem-MDP-recovery} below.
\end{remark}

\section{\large SEMI-IMPLICIT TIME DISCRETISATION
AND ITS CONVERGENCE}
\label{sec-disc}

To prove existence of local solutions,
we use
a constructive method relying on a suitable time discretisation and
the weak compactness of level sets of the minimization problems
arising at each time level. When further discretised in space, it will later
in Sect.~\ref{sect-BEM} yield a computer implementable efficient
algorithm. Let us summarize the assumption on the data of the 
original continuous problem:
\begin{subequations}\label{ass}\begin{align}\label{ass-C}
&\mathbb C(\cdot),\mathbb H\in\R^{d\times d\times d\times d}\text{ positive definite, symmetric},\ \ \mathbb C:[0,1]\to\R^{d\times d\times d\times d}
\text{ continuous},
\\&a,\,b,\,\kappa_1,\kappa_2
>0,\ \ S\subset\DEV\text{ convex, bounded, closed,\ \ int\,$S\ni0$},
\\\label{ass-BC}
&\wD\in W^{1,1}(I;W^{1/2,2}(\GD;\R^d)),\ \ 
\\[-.3em]
&g\in W^{1,1}(I;L^p(\Omega;\R^d))\ \ \ \text{ with }
\ \ p\begin{cases}>1&\text{for }d=2,\\[-.3em]
=2d/(d{+}2)\!&\text{for }d\ge3\end{cases}\displaybreak
\\[-.1em]
&f\in W^{1,1}(I;L^p(\GN;\R^d))\ \ \ \text{ with }
\ p\begin{cases}>1&\text{for }d=2,\\[-.3em]
=2{-}2/d\ \ \ &\text{for }d\ge3,\end{cases}
\\\label{ass-IC}
&(u_0,\pi_0,\zeta_0)\in H^1(\Omega;\R^d){\times}H^1(\Omega;\DEV){\times}
W^{1,r}(\Omega).
\end{align}\end{subequations}
The qualification \eqref{ass-BC} allows for an extension
$\uD$ of $\wD$ which belongs to $W^{1,1}(I;H^1(\Omega;\R^d))$;
in what follows, we will consider some extension with this property.

For the mentioned time discretisation, we use an equidistant partition
of the time interval $I=[0,T]$ with a time step $\tau>0$, assuming
$T/\tau\in\N$, and denote $\{u_\tau^k\}_{k=0}^{T/\tau}$ an approximation
of the desired values $u(k\tau)$, and similarly $\zeta_\tau^k$ is to
approximate $\zeta(k\tau)$, etc.

We use a decoupled semi-implicit time discretisation with the
\emph{fractional steps} based on the splitting
of the state variables governed by the separately-convex
character of $\scrE(t,\cdot,\cdot,\cdot)$. This will make the
numerics considerably easier than any other splitting and simultaneously
may lead  to a physically relevant solutions governed rather by
stresses (if the maximum-dissipation principle  holds at least approximately
in the sense of Remark~\ref{rem-DMP} below) than by energies and will prevent 
too-early debonding, as already announced in Section~\ref{sect-LS}.
More specifically, exploiting the convexity of both
$\scrE(t,\cdot,\cdot,\zeta)$ and $\scrE(t,u,\pi,\cdot)$ and 
the additivity $\scrR=\scrR_1(\DT\pi)+\scrR_2(\DT\zeta)$,
this splitting will be considered as $(u,\pi)$
and $\zeta$. This yields alternating convex minimization.
Thus, for $(\pi_\tau^{k-1},\zeta_\tau^{k-1})$ given, we obtain two
minimization problems
\begin{subequations}\label{dam-small-II-LS-min}
\begin{align}\label{minimize-1}
\left.\begin{array}{ll}
\text{minimize}&\scrE_\tau^k(u,\pi,\zeta_\tau^{k-1})+\scrR_1(\pi{-}\pi_\tau^{k-1})
\\\text{subject to}&(u,\pi)\in H^1(\Omega;\R^d)\times H^1(\Omega;\DEV),
\ \ u|_{\GD}^{}=0,
\end{array}\ \right\}
\intertext{with $\scrE_\tau^k:=\scrE(k\tau,\cdot,\cdot,\cdot)$
and, denoting the unique solution as $(u_\tau^k,\pi_\tau^k)$,}
\label{minimize-2}
\left.\begin{array}{ll}
\text{minimize}&\scrE_\tau^k(u_\tau^k,\pi_\tau^k,\zeta)
+\scrR_2(\zeta{-}\zeta_\tau^{k-1})
\\\text{subject to}&\zeta\in W^{1,r}(\Omega),\ \ 0\le\zeta\le1,
\end{array}\hspace{8em}\right\}
\end{align}
\end{subequations}
and denote its (possibly not unique) solution by $\zeta_\tau^k$.
Existence of the discrete solutions $(u_\tau^k,\pi_\tau^k,\zeta_\tau^k)$ is 
straightforward by the mentioned compactness arguments. 

We define the piecewise-constant interpolants
\begin{align}\label{def-of-interpolants}
\left.\begin{array}{l}
\baru_\tau(t)= u_\tau^k\,\ \ \ \&\ \ \ \underline u_\tau(t)= u_\tau^{k-1},
\\[.2em]\barpi_\tau(t)=\pi_\tau^k\ \ \ \&\ \ \ \underline\pi_\tau(t)=\pi_\tau^{k-1},
\\[.2em]\barzeta_\tau(t)=\zeta_\tau^k\,\ \ \ \&\ \ \
\underline\zeta_\tau(t)=\zeta_\tau^{k-1},
\\[.2em]\bar\scrE_\tau(t,u,\pi,\zeta)=\scrE_\tau^k(u,\pi,\zeta)
\end{array}\ \right\}\text{ for }(k{-}1)\tau<t\le k\tau.
\end{align}
Later in Remark~\ref{rem-DMP}, we will also use the piecewise affine
interpolants
\begin{align}\label{def-of-interpolants+}
\left.\begin{array}{l}
\pi_\tau(t)=\frac{t-(k{-}1)\tau}\tau\pi_\tau^k
+\frac{k\tau-t}\tau\pi_\tau^{k-1},
\\[.2em]\zeta_\tau(t)=\frac{t-(k{-}1)\tau}\tau\zeta_\tau^k
+\frac{k\tau-t}\tau\zeta_\tau^{k-1}
\end{array}\ \right\}\text{ for }(k{-}1)\tau<t\le k\tau.
\end{align}

The important attribute of the discretisation \eqref{dam-small-II-LS-min}
is also its numerical stability and satisfaction of a suitable discrete analog
of \eqref{def-LS}, namely:

\begin{proposition}[Stability of the time discretisation]
Let \eqref{ass} hold and, in terms of the interpolants 
\eqref{def-of-interpolants}, $(\baru_\tau,\barpi_\tau,\barzeta_\tau)$ be 
an approximate solution obtained by \eqref{dam-small-II-LS-min}.
Then, the following a-priori estimates hold
\begin{subequations}\label{dam-ls-estimates}
\begin{align}\label{dam-ls-estimate1}
&\big\|\baru_\tau\big\|_{ L^\infty(I; H^1(\Omega;\R^d))}
\le C,&&
\\\label{dam-ls-estimate3}
&\big\|\barpi_\tau\big\|_{ L^\infty(I; H^1(\Omega;\DEV))
\cap{\rm BV}(I; L^1(\Omega;\DEV))}\le C,
\\\label{dam-ls-estimate2}
&\big\|\barzeta_\tau\big\|_{ L^\infty(\Omega)
\cap{\rm BV}(I; L^1(\Omega))}\le C.
&&
\end{align}\end{subequations}
Moreover, the obtained approximate solution satisfies for any 
$t\in I{\setminus}\{0\}$ the (weakly formulated) Euler-Lagrange equation for 
the displacement:
\begin{subequations}\label{dam-ls}
\begin{align}
&
\scrE_u'\big(t_\tau,\baru_\tau(t),\barpi_\tau(t),\underline\zeta_\tau(t)\big)
=0,
\label{dam-ls-VI}
\intertext{with $t_\tau\!:=\min\{k\tau\!\ge\!t;\ k\!\in\!\N\}$,
two separate semi-stability conditions for $\barzeta_\tau$ and $\barpi_\tau$:}
&\forall\widetilde\pi\!\in\! H^1(\Omega;\DEV):
\label{dam-ls-semi-stab1}
\
\scrE\big(t_\tau,\baru_\tau(t),\barpi_\tau(t),\underline\zeta_\tau(t)\big)
\le\scrE\big(t_\tau,\baru_\tau(t),\widetilde\pi,\underline\zeta_\tau(t)\big)
+\scrR_1\big(\widetilde\pi{-}\barpi_\tau(t)\big),
\\\nonumber
&\forall\widetilde\zeta\!\in\! W^{1,r}(\Omega),\ 
0\!\le\!\widetilde\zeta\!\le\!
1:\ \,
\\
\label{dam-ls-semi-stab2}
&\hspace{8.3em}
\scrE\big(t_\tau,\baru_\tau(t),\barpi_\tau(t),\barzeta_\tau(t)\big)
\le\scrE\big(t_\tau,\baru_\tau(t),\barpi_\tau(t),\widetilde\zeta\big)
+\scrR_2\big(\widetilde\zeta{-}\barzeta_\tau(t)\big),
\intertext{and, for all $0\le t_1<t_2\le T$ of the 
form $t_i=k_i\tau$ for some $k_i\!\in\!\N$, the energy (im)balance:}
\nonumber&
\scrE\big(t_2,\baru_\tau(t_2),\barpi_\tau(t_2),\barzeta_\tau(t_2)\big)
+\Diss_{\scrR_1}^{}\big(\barpi_\tau;[t_1,t_2]\big)
+\Diss_{\scrR_2}\big(\barzeta_\tau;[t_1,t_2]\big)
\\[-.3em]\label{dam-ls-semi-engr}&\hspace{10em}
\le
\scrE\big(t_1,\baru_\tau(t_1),\barpi_\tau(t_1),\barzeta_\tau(t_1)\big)
+\!\int_{t_1}^{t_2}\!\!
\scrE_t'\big(t,\baru_\tau(t),\barpi(t),\barzeta(t))\,\d t.
\end{align}\end{subequations}
\end{proposition}

\noindent{\it Sketch of the proof.}
Writing optimality condition for \eqref{minimize-1} in terms of $u$,
one arrives at \eqref{dam-ls-VI}, and comparing the value of
\eqref{minimize-1} at $(u_\tau^k,\pi_\tau^k)$ with its value at
$(u_\tau^k,\widetilde\pi)$ and using the degree-1 homogeneity of $\scrR_1$,
one arrives at \eqref{dam-ls-semi-stab1}.

Comparing the value of \eqref{minimize-2} at $\zeta_\tau^k$ with its
value at $\widetilde\zeta$ and using the degree-1 homogeneity of
$\scrR_2$, one arrives at \eqref{dam-ls-semi-stab2}.

In obtaining \eqref{dam-ls-semi-engr},
we compare the value of \eqref{minimize-1} at the minimizer
$(u_\tau^k,\pi_\tau^k)$ with the value at $(u_\tau^{k-1},\pi_\tau^{k-1})$
and  the value of \eqref{minimize-2}  at the minimizer $\zeta_\tau^k$ with
the value at $\zeta_\tau^{k-1}$ and we benefit from the cancellation of the 
terms $\pm\scrE(k\tau,u_\tau^k,\pi_\tau^k,\zeta_\tau^{k-1})$.
We also use the discrete by-part integration (=\,summation) for the 
$\scrE_t'$-term.

Then, using \eqref{dam-ls-semi-engr} for $t_1=0$ and the coercivity of 
$\scrE(t,\cdot,\cdot,\cdot)$ due to the assumptions \eqref{ass}, we obtain 
also the a-priori estimates \eqref{dam-ls-estimates}.
$\hfill\Box$

\medskip

The cancellation effect mentioned in the above proof is typical in 
fractional-step methods, cf.\ e.g.\ \cite[Remark~8.25]{Roub13NPDE}. 
Further, note that
\eqref{dam-ls} is of a similar form as \eqref{def-LS} and is
thus prepared to make a limit passage for $\tau\to0$:

\begin{proposition}[Convergence towards local solutions]
\label{ch4:prop-dam-plast}
\index{local solution!to plasticity with damage}
Let  \eqref{ass} hold
and let 
$(\baru_\tau,\barpi_\tau,\barzeta_\tau)$ be an approximate solution 
obtained by the semi-implicit formula 
\eqref{dam-small-II-LS-min}. Then there exists a subsequence (indexed
again by $\tau$ for notational simplicity) and 
$u\in{\rm B}([0,T];H^1(\Omega;\R^d))$ and 
$\pi\in{\rm B}([0,T];H^1(\Omega;\DEV))\cap{\rm BV}([0,T];L^1(\Omega;\DEV))$
and $\zeta\in{\rm B}([0,T];W^{1,r}(\Omega))\cap{\rm BV}([0,T];L^1(\Omega))$
such that
\begin{subequations}\label{dam-plast-conv}
\begin{align}\label{dam-plast-conv-u}
&\baru_\tau(t)\to u(t)&&\text{in } H^1(\Omega;\R^d)
&&\hspace*{-2.5em}\text{for all }t\in[0,T],&&
\\\label{dam-plast-pi}&
\barpi_\tau(t)\to\pi(t)&&\text{in } H^1(\Omega;\DEV)
&&\hspace*{-2.5em}\text{for all }t\in[0,T],&&
\\\label{dam-plast-zeta}&
\barzeta_\tau(t)\to\zeta(t)&&\text{in }
 W^{1,r}(\Omega)
&&\hspace*{-2.5em}\text{for all }t\in[0,T].&&
\end{align}
\end{subequations}
Moreover, any $(u,\pi,\zeta)$ obtained by this way is a local solution 
to the damage/plasticity problem in that sense of Definition~\ref{def-loc-sln}.
\end{proposition}

\begin{proof}
By a (generalized) Helly's selection principle, cf.\ also
e.g.\ \cite{Miel11DEMF,MieRou15RIST}, we choose a subsequence and
$\pi\in{\rm B}([0,T];H^1(\Omega;\DEV))\cap{\rm BV}([0,T];L^1(\Omega;\DEV))$
and $\zeta,\,\zeta_*\in{\rm B}([0,T];W^{1,r}(\Omega))$ $\cap$ 
${\rm BV}([0,T];L^1(\Omega))$
so that
\begin{subequations}\label{eq3:LS-conv-to-z}\begin{align}
&\barpi_\tau(t)\weak\pi(t)&&&&
&&\text{ in }\ H^1(\Omega;\DEV)\ \text{ for all }\ t\!\in\![0,T],
\\
&\barzeta_\tau(t)\weak\zeta(t)&&\&&&\underline\zeta_\tau(t)\weak\zeta_*(t)
&&\text{ in }\ W^{1,r}(\Omega)\ \ \ \ \ \ \text{ for all }\ \ t\!\in\![0,T].
\end{align}\end{subequations}
Now, for a fixed $t\in [0,T]$, by Banach's selection principle, we select
(for a moment) further subsequence so that
\begin{align}\label{weak-conv-u}
\baru_\tau(t)\weak u(t)\quad\text{ in }\ H^1(\Omega;\R^d).
\end{align}
We further use that $\baru_\tau(t)$ minimizes
$\scrE(t_\tau,\cdot,\barpi_\tau(t),\underline\zeta_\tau(t))$
with $t_\tau:=\min\{k\tau\ge t;\ k\in\N\}$.
Obviously, $t_\tau\to t$ for $\tau\to0$
and, by the weak-lower-semicontinuity argument, we can easily
see that $u(t)$ minimizes the strictly convex functional
$\scrE(t,\cdot,\zeta_*(t),\pi(t))$; this is indeed simple to 
prove due to the compactness in both $\pi$ and $\zeta$ due to the gradient 
theories involved.
Thus $u(t)$ is determined uniquely so that, in fact, we did not need to
make further selection of a subsequence, and this procedure
can be performed for any $t$  by using the same subsequence already
selected for \eqref{eq3:LS-conv-to-z}. Also, $u:[0,T]\to H^1(\Omega;\R^d)$ is 
measurable because $\pi$ and $\zeta_*$ are measurable, and 
$\scrE_u'(t,u(t),\pi(t),\zeta_*(t))\!=\!0$ for all $t$.

The key ingredient is improvement of the weak convergence
\eq{eq3:LS-conv-to-z} and \eq{weak-conv-u} for the strong 
convergence. For the strong convergence in $u$ and $\pi$, we use 
the uniform convexity 
of the quadratic form induced by $\mathbb{C}(\zeta)$, $\mathbb{H}$, and 
$\kappa_1$ with the information we have at disposal from 
\eqref{dam-ls-semi-stab1}
leading, when using the abbreviation $e^{}_\mathrm{el}=e(u{-}\uD)-\pi$ and 
$\bare^{}_{\mathrm{el},\tau}=e(\baru_\tau{-}\uDtau)-\barpi_\tau$, to the
estimate:
\begin{align}\nonumber
\nonumber&\hspace*{-.1em}\int_{\Omega}\!
\mathbb{C}(\ul\zeta_\tau(t))\big(\bare^{}_{\mathrm{el},\tau}(t){-}e^{}_\mathrm{el}(t)\big)
:\big(\bare^{}_{\mathrm{el},\tau}(t){-}e^{}_\mathrm{el}(t)\big)
\\[-.2em]\nonumber&\quad
+\mathbb{H}\big(\barpi_\tau(t){-}\pi(t)\big):\big(\barpi_\tau(t){-}\pi(t)\big)
+\frac{\kappa_1}2\big|\nabla\barpi_\tau(t){-}\nabla\pi(t)\big|^2\,\d x
\\[-.2em]\nonumber&\
\le\int_{\Omega}\!\!-\mathbb{C}(\ul\zeta_\tau(t))e^{}_\mathrm{el}(t):
\big(\bare^{}_{\mathrm{el},\tau}(t){-}e^{}_\mathrm{el}(t)\big)
-\big(\mathbb{H}\pi(t){-}\barxi_\tau(t)\big):\big(\barpi_\tau(t){-}\pi(t)\big)
\\[-.2em]
\nonumber
&\quad
+\frac{\kappa_1}2\nabla\pi(t):\nabla\big(\barpi_\tau(t){-}\pi(t)\big)
-\barf_\tau(t){\cdot}(\baru_\tau(t){-}u(t))
\,\d x-\int_{\GN}\!\barg_\tau(t){\cdot}(\baru_\tau(t){-}u(t))\,\d S\to0
\end{align}
where we use some $\barxi_\tau(t)\in\pl\delta_S^*(\barpi_\tau(t))$ which 
solves at time $t$ in the weak sense the discrete plastic flow-rule 
$\barxi_\tau+\mathbb{H}\barpi_\tau-{\rm dev}\,\barsigma_\tau=\kappa_1\Delta\barpi_\tau$
with $\barsigma_\tau=\mathbb{C}(\ul\zeta_\tau)\bare^{}_{\mathrm{el},\tau}$. 
Thus we proved 
\begin{align}\nonumber
\bare^{}_{\mathrm{el},\tau}(t)\to e^{}_\mathrm{el}(t)
\ \ \text{ strongly in }\ L^2(\Omega;\SYM)
\end{align} together with \eq{dam-plast-pi}.
Realizing that 
$e(\baru_\tau(t))=e(\uDtau(t))+\barpi_\tau(t)+\bare^{}_{\mathrm{el},\tau}(t)$,
we obtain also $e(\baru_\tau(t))\to e(u(t))$ strongly in $L^2(\Omega;\SYM)$,
and thus also \eq{dam-plast-conv-u}.
Note that we exploited the 
gradient theory for plasticity which ensures that the sequence 
$(\barxi_\tau)^{}_{\tau>0}$, which is 
bounded in $L^\infty(\Omega;\DEV)$ because the plastic domain $S\subset\DEV$ 
is bounded, is relatively compact in $H^1(\Omega;\DEV)^*$ so that 
the term $\int_{\Omega}\barxi_\tau(t):(\barpi_\tau(t){-}\pi(t))\,\d x$
indeed converges to zero because $\barpi_\tau(t)\weak\pi(t)$ in 
$H^1(\Omega;\DEV)$.

The convergence \eq{dam-plast-zeta} can be proved by
using the uniform-like monotonicity of the set-valued mapping
$\zeta\mapsto\pl\delta_{[0,1]}^{}(\zeta)
-\kappa_2\,\mathrm{div}(|\nabla\zeta|^{r-2}\nabla\zeta):
W^{1,r}(\Omega)\rightrightarrows W^{1,r}(\Omega)^*$. Analogously
to \eqref{plast-dam13}, we can write the discrete damage flow rule 
after the shift \eqref{subst-zero-Dirichlet} as
\begin{subequations}\label{dis-dam-flow-rule}\begin{align}
\barxi_{\mathrm{dam},\tau}
+\mathbb{C}'(\ul\zeta_\tau)\bare^{}_{\mathrm{el},\tau}:\bare^{}_{\mathrm{el},\tau}
=\kappa_2\,\mathrm{div}(|\nabla\barzeta_\tau|^{r-2}\nabla\barzeta_\tau)
-\bareta_\tau\qquad\qquad\qquad
\\\qquad\qquad
\text{ with some}\ \barxi_{\mathrm{dam},\tau}\!\in\!\pl\delta_{[-a,b]}^*(\DT\zeta_\tau)
\ \text{ and }\ \bareta_\tau\!\in\!\pl\delta_{[0,1]}^{}(\barzeta_\tau)
\end{align}\end{subequations}
with the boundary condition $\nabla\barzeta_\tau\cdot\vec{n}=0$ on $\Sigma$; in 
\eqref{dis-dam-flow-rule}, $\barxi_{\mathrm{dam},\tau}$ and 
$\bareta_\tau$ are considered piece-wise constant in time, consistently 
with our bar-notation. An important fact is that $\barxi_{\mathrm{dam},\tau}(t)$
is valued in $[-b,a]$ and hence a-priori bounded in $L^\infty(\Omega)$; here
we vitally exploited the concept of possible (small) healing allowed.
We can rely on $\barxi_{\mathrm{dam},\tau}(t)\weaks\xi_{\mathrm{dam}}(t)$ in 
$L^\infty(\Omega)$
for some $t$-dependent subsequence and some $\xi_{\mathrm{dam}}(t)$. Using that 
$\mathbb{C}'(\ul\zeta_\tau(t))\bare^{}_{\mathrm{el},\tau}(t):
\bare^{}_{\mathrm{el},\tau}(t)$ is bounded and, due to 
(\ref{dam-plast-conv}a,b), even has been proved converging 
in $L^1(\Omega)$ which is a subspace of $W^{1,r}(\Omega)^*$ because
$r>d$ is considered. By the standard theory for monotone variational 
inequalities, we can pass to the limit in \eqref{dis-dam-flow-rule} at 
time $t$ to obtain, in the weak formulation,
\begin{align}\label{dis-dam-flow-rule-weak}
\xi_{\mathrm{dam}}(t)
+\mathbb{C}'(\zeta_*(t))e^{}_{\mathrm{el}}(t){:}e^{}_{\mathrm{el}}(t)
=\kappa_2\,\mathrm{div}(|\nabla\zeta(t)|^{r-2}\nabla\zeta(t))
-\eta(t)\ \ \text{ with }\ \eta(t)\!\in\!\pl\delta_{[0,1]}^{}(\zeta(t)).
\end{align}
Then, at any $t$, we can estimate
\begin{align}\nonumber
&\kappa_2\limsup_{k\to\infty}\big(\|\nabla \barzeta_\tau(t)\|_{L^r(\Omega;\R^d)}^{r-1}\!
-\|\nabla \zeta(t)\|_{L^r(\Omega;\R^d)}^{r-1}\big)
\big(\|\nabla \barzeta_\tau(t)\|^{}_{L^r(\Omega;\R^d)}\!
-\|\nabla \zeta(t)\|^{}_{L^r(\Omega;\R^d)}\big)
\\&\nonumber\quad\le
\limsup_{k\to\infty}\int_\Omega\kappa_2\big(|\nabla \barzeta_\tau(t)|^{r-2}\nabla \barzeta_\tau(t)
-|\nabla \zeta(t)|^{r-2}\nabla \zeta(t)\big){\cdot}\nabla(\barzeta_\tau(t){-}\zeta(t))
\\[-.3em]&\nonumber\hspace{20em}
+(\bareta_\tau(t){-}\eta(t))(\barzeta_\tau(t){-}\zeta(t))\dd x
\\&\nonumber\quad
=\lim_{k\to\infty}\int_\Omega
\mathbb{C}'(\ul\zeta_\tau(t))\bare^{}_{\mathrm{el},\tau}(t)
:\bare^{}_{\mathrm{el},\tau}(t)\big(\barzeta_\tau(t){-}\zeta(t)\big)
-\kappa_2\,|\nabla \zeta(t)|^{r-2}\nabla \zeta(t){\cdot}\nabla(\barzeta_\tau(t){-}\zeta(t))
\\[-.3em]\label{conv-of-zeta}&\hspace{18em}
-(\xi_{\mathrm{dam}}(t){+}\eta(t))(\barzeta_\tau(t){-}\zeta(t))\dd x=0
\end{align}
where the last equality has exploited \eqref{dis-dam-flow-rule-weak}.
The important fact used for \eq{conv-of-zeta} is that 
\begin{align}\label{eq4:dam-plast-conv-zeta}
\mathbb{C}'(\ul\zeta_\tau(t))\bare^{}_{\mathrm{el},\tau}(t)
:\bare^{}_{\mathrm{el},\tau}(t)
\big(\barzeta_\tau(t){-}\zeta(t)\big)\to0\ \ \text{ weakly in }\ L^1(\Omega);
\end{align}
in fact, this convergence is even strong when realizing that
$\barzeta_\tau(t)\to\zeta(t)$ in $L^\infty(\Omega)$, for 
which again $r>d$ is exploited. From this, \eqref{dam-plast-zeta} follows.
Thus, from \eqref{conv-of-zeta} we can see that 
$\|\nabla\barzeta_\tau(t)\|_{L^r(\Omega;\R^d)}\to\|\nabla\zeta(t)\|_{L^r(\Omega;\R^d)}$
and, from uniform convexity of the Lebesgue space $L^r(\Omega;\R^d)$,
we eventually obtain \eqref{dam-plast-zeta}.
Actually, the specific value $\xi_{\mathrm{dam}}(t)$
of the limit of (a $t$-dependent subsequence of) 
$\{\xi_{\mathrm{dam},\tau}(t)\}_{\tau>0}$
which is surely precompact in $W^{1,r}(\Omega)^*$ is not important 
and thus \eqref{dam-plast-zeta} holds for the originally selected 
subsequence, too.

Having the strong convergences \eq{dam-plast-conv} proved, the limit
passage from \eqref{dam-ls} towards \eq{def-LS} is simple. In particular, 
by continuity of both BV-functions $\zeta(\cdot)$ and $\zeta_*(\cdot)$
on $[0,T]{\setminus}J$ for some at most countable set $J$, we have also 
$\zeta_*(t)=\zeta(t)$ at any $t$ except at most countable the set $J$.
\end{proof}

\begin{remark}[{\sl Approximate maximum-dissipation principle}]\label{rem-DMP}
\upshape
One can devise the discrete analog of the integrated maximum-dissipation
principle \eqref{IMDP} straightforwardly for the left-continuous
interpolants \eqref{def-of-interpolants},
required however to hold only asymptotically.
More specifically, in analog to \eqref{IMDP} formulated
equivalently for all $[0,t]$ instead of $[t_1,t_2]$, one can expect
an {\it Approximate Maximum-Dissipation Principle} (AMDP) in the form
\begin{subequations}\label{AMDP}
\begin{align}\label{AMDP-a}
&\!\int_0^t\!\barxi_{\mathrm{plast},\tau}\,\d\barpi_\tau
\ \stackrel{\mbox{\bf ?}}{\sim}\ \Diss_{\scrR_1}^{}(\barpi_\tau;[0,t])
\ \ \ \text{ with }\ \ \barxi_{\mathrm{plast},\tau}
=-\big[\bar{\scrE}_\tau\big]_\pi'(\cdot,\baru_\tau,\barpi_\tau,\underline\zeta_\tau),
\\\label{AMDP-b}
&\!\int_0^t\!\barxi_{\mathrm{dam},\tau}\,\d\barzeta_\tau
\ \stackrel{\mbox{\bf ?}}{\sim}\ \Diss_{\scrR_2}^{}(\barzeta_\tau;[0,t])
\ \ \ \text{ for some }\ \ \barxi_{\mathrm{dam},\tau}
\!\in\!
-\partial_\zeta\bar{\scrE}_\tau(\cdot,\baru_\tau,\barpi_\tau,\barzeta_\tau),
\end{align}\end{subequations}
or, analogously to \eqref{IMDP+}, 
\begin{align}\label{AMDP+}
&\!\int_0^t\!\barxi_{\tau}\,\d\barz_\tau
\,\stackrel{\mbox{\bf ?}}{\sim}\,\Diss_{\scrR}^{}(\barz_\tau;[0,t])
\  \text{ with some }\ \barxi_{\tau}
\in\big\{{-}\big[\bar{\scrE}_\tau\big]_\pi'(\cdot,\baru_\tau,\barpi_\tau,\underline\zeta_\tau)\big\}\times
-\partial_\zeta^{}\bar{\scrE}_\tau(\cdot,\baru_\tau,\barpi_\tau,\barzeta_\tau),
\end{align}
where the integrals are again the lower Moore-Pollard-Stieltjes 
integrals as in \eqref{IMDP} and where $\bar{\scrE}_\tau(\cdot,u,\pi,\zeta)$ 
is the left-continuous piecewise-constant interpolant of the values 
$\scrE(k\tau,u,\pi,\zeta)$, $k=0,1,...,T/\tau$.
Moreover, ''${\stackrel{\mbox{\footnotesize\bf ?}}{\sim}}$'' in \eqref{AMDP}
means that the equality holds possibly only asymptotically for $\tau\to0$ but
even this is rather only desirable and not always valid.
Anyhow, loadings which, under given geometry of the specimen,
lead to  rate-independent  slides  where the solution
is absolutely continuous will always comply with AMDP \eqref{AMDP}.
Also, some finite-dimensional examples  of ``damageable springs''
in \cite{MieRou15RIST,Roub15MDLS} show that this AMDP can detect too
early rupturing local solutions (in particular the energetic ones)
while it generically holds for solutions obtained by the
algorithm \eqref{dam-small-II-LS-min}. 
Generally speaking, \eqref{AMDP} should rather be a-posteriori checked to 
justify the (otherwise not physically based) simple and numerically efficient
fractional-step-type semi-implicit algorithm \eqref{dam-small-II-LS-min}
from the perspective of the stress-driven solutions in particular
situations and possibly to provide a valuable information
that can be exploited to adapt time or space discretisation towards
better accuracy in \eqref{AMDP} and thus close towards the stress-driven
scenario. Actually, for the piecewise-constant interpolants, we can simply 
evaluate the integrals explicitly, so that AMDP \eqref{AMDP+} reads
\begin{align}\nonumber
&\!\!
\sum_{k=1}^K\!\int_{\Omega}\delta_S^*(\pi(t_{j-1}){-}\pi(t_j))+
a\big(\zeta_\tau^k{-}\zeta_\tau^{k-1}\big)^-\!
+b\big(\zeta_\tau^k{-}\zeta_\tau^{k-1}\big)^+\,\d x
\\[-1.6em]&\hspace{13em}
-
\!\big\langle\xi_{\mathrm{plast},\tau}^{k-1},\pi_\tau^k{-}\pi_\tau^{k-1}\big\rangle
-\big\langle\xi_{\mathrm{dam},\tau}^{k-1},\zeta_\tau^k{-}\zeta_\tau^{k-1}\big\rangle
\ \stackrel{\mbox{\large\bf ?}}{\le}\eps_\tau\searrow\ 0
\label{AMDP++}
\\[.2em]&\nonumber\qquad\quad\text{where }
\ \ 
\xi_{\mathrm{plast},\tau}^k=
-\big[\scrE_\tau^k\big]_\pi'(u_\tau^k,\pi_\tau^k,\zeta_\tau^{k-1})
\ \text{ and }\ 
\xi_{\mathrm{dam},\tau}^k\in
-\partial_\zeta^{}\scrE_\tau^k(u_\tau^k,\pi_\tau^k,\zeta_\tau^k)
\end{align}
for some $\eps_\tau\searrow\ 0$,
where $K=\max\{k\!\in\!\mathbb N;\ k\tau\le t\}$. Notably, in contrast to 
\eqref{IMDP} and \eqref{IMDP+}, the AMDP \eqref{AMDP} and \eqref{AMDP+} are 
equivalent to each other as the limsup's (cf.\ the definition 
\eqref{def-of-RS-int}) in all involved integrals is attained on the 
equidistant partitions with the time step $\tau$ and the ``inf'' in the 
Darboux sums is redundant. Evaluating the 
dualities,
\eqref{AMDP++} can be written more explicitly as 
$\int_\Omega R_\tau^K\,\d x{\stackrel{\mbox{\footnotesize\bf ?}}{\ \le\ }}\eps_\tau\searrow\ 0$
with the residuum 
\begin{align}\nonumber
R_\tau^K&:=
\sum_{k=1}^K\bigg(\delta_S^*(\pi(t_{j-1}){-}\pi(t_j))+a\big(\zeta_\tau^k{-}\zeta_\tau^{k-1}\big)^-\!
+b\big(\zeta_\tau^k{-}\zeta_\tau^{k-1}\big)^+
\\[-.1em]&\nonumber\ \ \
-\Big(\mathbb{C}(\zeta_\tau^{k-2})\big(\pi_\tau^{k-1}\!-e(u_\tau^{k-1}{+}\uDtaukk)\big)
+\mathbb{H}\pi_\tau^{k-1}\Big):\big(\pi_\tau^k{-}\pi_\tau^{k-1}\big)
\\[-.1em]&\nonumber\ \ \
-\Big(\frac12\mathbb{C}'(\zeta_\tau^{k-1})
\big(e(u_\tau^{k-1}{+}\uDtaukk){-}\pi_\tau^{k-1}\big)
{:}\big(e(u_\tau^{k-1}{+}\uDtaukk){-}\pi_\tau^{k-1}\big)
+\xi_{\mathrm{const},\tau}^{k-1}\Big)\big(\zeta_\tau^k{-}\zeta_\tau^{k-1}\big)
\\[-.2em]&\qquad\qquad\qquad
-\kappa_1\nabla\pi_\tau^{k-1}\Vdots\nabla\big(\pi_\tau^k{-}\pi_\tau^{k-1}\big)
-\kappa_2\big|\nabla\zeta_\tau^{k-1}\big|^{r-2}
\nabla\zeta_\tau^{k-1}{\cdot}\nabla\big(\zeta_\tau^k{-}\zeta_\tau^{k-1}\big)\bigg)
\label{rhs-to-evaluate}\end{align}
with some multiplier $\xi_{\mathrm{const},\tau}^k\!\in\!N_{[0,1]}^{}(\zeta_\tau^k)$
and with $\zeta_\tau^{k-2}$ for $k=1$ equal to $\zeta_0$. Note that $R_\tau^K$ 
cannot be guaranteed non-negative pointwise on $\Omega$, only
their integrals over $\Omega$ are non-negative. 
One can a-posteriori check the residua depending on $t$ 
or possibly also on space, cf.\ also \cite{RoPaMa??LSAR,VoMaRo14EMDL} for a 
surface variant of such a model or \ Figures~\ref{fig-movie-asym}--\ref{fig:residua} below.
\end{remark}

\section{\large IMPLEMENTATION OF THE DISCRETE MODEL}
\label{sect-BEM}

To implement the model computationally, we need to make a {\it spatial discretisation} of the 
variables from the semi-implicit time discretization of Section 4. Essentially,
we apply conformal Galerkin (or also called Ritz) method to the 
minimization problems \eqref{minimize-1} and \eqref{minimize-2} which are then 
restricted to the corresponding finite-dimensional subspaces. These 
subspaces are constructed by the {\it finite-element method} (FEM), and the 
solution thus obtained is denoted by 
$$q_{\tau h}^k:=\big(u_{\tau h}^k, \pi_{\tau h}^k, \zeta_{\tau h}^k\big),$$
with $h>0$ denoting the mesh size of the triangulation, let us denote it by $\mathscr T_h$,
of the domain $\Omega$ considered polyhedral here. By this way, we obtain also the
piecewise constant and the piecewise affine interpolants in time, denoted respectively
by $\baru_{\tau h}$ and $u_{\tau h}$, $\barpi_{\tau h}$ and $\pi_{\tau h}$, 
and eventually $\barzeta_{\tau h}$ and $\zeta_{\tau h}$.
The simplest option is to consider the lowest-order conformal  FEM, i.e.\
P1-elements for  $u$,  $\zeta$, and  $\pi$.
In Sect.~\ref{sect-simul}, 
only the case $d=2$ will be treated, so the previous analytical part have required $r>2$ 
and we make 
an (indeed small) shortcut by considering $r=2$. Moreover, we will not consider the
loading on $\GN$ so we put $f=0$.

The material is assumed isotropic with properties linearly dependent on damage.
The isotropic elasticity tensor is assumed as 
\begin{equation}\label{C-ass-implem}
\mathbb C_{ijkl} (\zeta):=\big[(\lambda_1{-}\lambda_0)\zeta+\lambda_0\big]\delta_{ij}\delta_{kl}  
+\big[(\mu_1{-}\mu_0)\zeta+\mu_0\big]\big(\delta_{ik}\delta_{jl}{+}\delta_{il}\delta_{jk}\big)               
\end{equation}
where $\lambda_1,\mu_1$ and $\lambda_0,\mu_0$ are two sets of Lam\'e parameters satisfying
$\lambda_1\ge\lambda_0\ge0$ and $\mu_1\ge\mu_0>0$. Here, $\delta$ denotes the Kronecker 
symbol. This choice implies that the elastic-moduli tensor 
is  positive-definite-valued (and therefore invertible). 
The elastic domain $S$ is assumed to satisfy 
\begin{equation}\label{S_def}
S=\big\{\sigma\in\DEV;\ |\sigma|\le\sY\big\},
\end{equation}
where $\sY>0$ is a given plastic yield stress. 
More specifically, the minimization problems 
\eqref{dam-small-II-LS-min} after spatial discretisation rewrite as
\begin{subequations}\label{minimization-implem}\begin{align}\nonumber
&(u_{\tau h}^k,\pi_{\tau h}^k) = 
\!\!\!\!\!\!\!\!\!\!\!\!
\argmin_{\substack{u\in W^{1,\infty}(\Omega;\R^d)\\\pi\in W^{1,\infty}(\Omega;{\DEV})\\u,\pi\text{ elementwise affine on }\mathscr T_h}}
\!\!\!\!\!\!\!\!\!\! 
\int_{\Omega}\bigg(\frac12\mathbb C(\zeta_{\tau h}^{k-1})\big(e(u{+}\uDtauhk){-}\pi\big):\big(e(u{+}\uDtauhk){-}\pi \big) 
\\[-2.3em]&\hspace{14em}+\frac12\mathbb{H}\pi{:}\pi+\frac{\kappa_1}2|\nabla\pi|^2\!
-g_{\tau h}^k{\cdot}u+ \sY | \pi{-}\pi_{\tau h}^{k-1}| \bigg)\dd x
,\!\!
\label{minimization-1-implem}\\\nonumber
&\zeta_{\tau h}^k   = 
\!\!\!\!\!\!\!\!\!\!
\argmin_{\substack{\zeta\in W^{1,\infty}(\Omega)\\0\le\zeta\le1\text{ on }\Omega\\\zeta\text{ elementwise affine on }\mathscr T_h}}
\!\!\!\!\!\!\!\!\!\! 
\int_{\Omega}\bigg(\frac12\mathbb C(\zeta)\big(e(u_{\tau h}^k{+}\uDtauhk){-}\pi_{\tau h}^k\big)
:\big(e(u_{\tau h}^k{+}\uDtauhk){-}\pi_{\tau h}^k\big) 
 \\[-2.3em] &\hspace{15.3em}
+\frac{\kappa_2}2|\nabla\zeta|^2\!+ a( \zeta{-} \zeta_{\tau h}^{k-1})^- + b( \zeta{-} \zeta_{\tau h}^{k-1})^+
\bigg)\dd x. 
\label{minimization-2-implem}\end{align}\end{subequations}

The damage problem \eqref{minimization-2-implem} represents a 
minimization of a nonsmooth but strictly convex functional. 
To facilitate its numerical solution, we still modify it a bit, namely
\def\ZZ{z}
\begin{subequations}\begin{align}\nonumber
&\!\!\!\!\!\!\!\!\!\!\!\!\!\!\!\!\!\!
\argmin
_{\substack{\zeta,\,\zeta^\vartriangle,\,\zeta^\triangledown\in W^{1,\infty}(\Omega)
\\0\le\zeta\le1\text{ on }\Omega
\\\zeta,\zeta^\vartriangle,\zeta^\triangledown\text{ elementwise affine on }\mathscr T_h}}\!\!\!\!\!
\int_{\Omega}\bigg(\,\frac12
\mathbb C (\zeta) \big(e(u_{\tau h}^k{+}\uDtauhk){-}\pi_{\tau h}^k\big)
:\big(e(u_{\tau h}^k{+}\uDtauhk){-}\pi_{\tau h}^k\big) 
\\[-2.3em]&\hspace{19.5em}
+\frac{\kappa_2}2|\nabla\zeta|^2+a\zeta^\vartriangle+b\zeta^\triangledown\bigg)\dd x,
\label{minimization-2-QP}
\\[-.1em]\label{minimization-2-QP+}
&\text{where }\ \ 
\zeta^\vartriangle = (\zeta{-}\zeta_{\tau h}^{k-1})^+\ \ \text{ and }
\ \ \zeta^\triangledown = (\zeta{-}\zeta_{\tau h}^{k-1})^-\ \text{ at all nodal points}.
\end{align}\end{subequations}
We used additional auxiliary `update' variables $\zeta^\vartriangle$ and 
$\zeta^\triangledown$ which are also considered
as P1-functions. This modification can also be understood as a certain 
specific numerical integration applied to the original minimization problem 
\eqref{minimization-2-implem}.
It should be noted that $\zeta$ and $\zeta_{\tau h}^{k-1}$ are 
P1-functions and, if we would require \eqref{minimization-2-QP+}
valid everywhere on $\Omega$, 
$\zeta^\vartriangle$ and $\zeta^\triangledown$ could not be P1-functions in 
general on elements where nodal values of $\zeta{-}\zeta_{\tau h}^{k-1}$ 
alternate signs. The important advantage of \eqref{minimization-2-QP+} 
required only at nodal points while at remaining points it is fulfilled only approximately 
(depending on $h$) is that \eqref{minimization-2-QP} actually represents a 
conventional {\it quadratic-programming} problem (QP) involving the linear and the box constraints
\begin{equation}\label{linear_and_box_constraints_QP}
\zeta = \zeta_{\tau h}^{k-1} + \zeta^\vartriangle  - \zeta^\triangledown, \qquad 
0\le\zeta^\vartriangle\le1-\zeta_{\tau h}^{k-1},
\qquad 
0\le\zeta^\triangledown\le\zeta_{\tau h}^{k-1}.
\end{equation}
A convex quadratic cost functional of this QP problem has only a positive-semidefinite Jacobian, 
since there are no Dirichlet boundary conditions on the damage variable $\zeta$. 
Note that the optimal 
pair $(\zeta^\vartriangle, \zeta^\triangledown)$ must satisfy $\zeta^\vartriangle\zeta^\triangledown {=} 0$ in all nodes, i.e.\ both 
variables cannot be positive. This can be easily seen by contradiction: 
If $\zeta^\vartriangle  \zeta^\triangledown {>} 0$ in some node, then a different pair 
$(\zeta^\vartriangle{-}\min\{\zeta^\vartriangle,\zeta^\triangledown\},\zeta^\triangledown{-}\min\{\zeta^\vartriangle,\zeta^\triangledown\})$ would again satisfy 
the constraints \eqref{linear_and_box_constraints_QP} but would provide a smaller 
energy value in \eqref{minimization-2-QP}.

As we have a-priori bounds of $\zeta$ in $W^{1,r}(\Omega)$ uniformly in $t$, $\tau$, 
and $h$ also if the modified problem \eqref{minimization-2-QP+} is considered
(disregarding that we used $r=2$ above), we have estimates also in H\"older spaces 
also for $\zeta^\vartriangle$ and $\zeta^\triangledown$ and
can show that the constraints \eqref{minimization-2-QP+} are valid 
everywhere on $\Omega$ in the limit for $h\to0$. Thus, an analogy of 
Proposition~\ref{ch4:prop-dam-plast} for a successive limit passage 
$h\to0$ and then $\tau\to0$ might be obtained, although it does not
have much practical importance for situations when $(h,\tau)\to(0,0)$
simultaneously.

A similar modification can be used also for \eqref{minimization-1-implem}.
In addition, one can then exploit 
the structure of the cost functional being the sum of a quadratic functional and 
a nonsmooth convex functional with the epi-graph having a ``ice-cream-cone'' shape. 
After introduction of auxiliary variables at each element, 
it can be transformed to a so-called second-order cone programming 
problem (SOCP), cf.\ \cite[Sect.\,3.6.3]{MieRou15RIST}, for which efficient
codes exist.

Other way is to use simply the quasi-Newton iterative method. 
This option was used also here.

\section{\large ILLUSTRATIVE 2-DIMENSIONAL EXAMPLES}
\label{sect-simul}

Finally, we demonstrate both the relevance of the model together with 
the solution concept from Sect.\,\ref{sec-model}-\ref{sect-LS}
and the efficiency and convergence of the discretisation scheme from 
Sect.\,\ref{sec-disc} together with the implementation from 
Sect.\,\ref{sect-BEM} on a two-dimensional example. 

\medskip\noindent{\sl The material}: We consider an isotropic homogeneous material
with the elastic properties given by Young's modulus $E_{_\mathrm{Young}}\!\!\!\!=27\,$GPa 
and Poisson's ratio $\nu=0.2$ in the non-damaged state, which means that 
the elastic-moduli tensor in the form \eqref{C-ass-implem} takes
$\lambda_1 = 7.5\,\mbox{GPa}$ and $\mu_1=11.25\,\mbox{GPa}$, while the damaged material uses 
$10^7$-times smaller moduli, i.e.\ $\lambda_0= 750\,\mbox{Pa}$ and $\mu_0=112.5\,\mbox{Pa}$
in \eqref{C-ass-implem}.
The yield stress from \eqref{S_def} and the kinematic hardening parameter are chosen as 
$\sigma_{\mbox{\tiny\rm Y}}=2\,\mbox{MPa}$ and $\mathbb H=(E_{_\mathrm{Young}}/20)\mathbb I$.
The activation energy for damage is $a=1.2\,\mbox{kPa}$ 
and the damage length-scale coefficient is $\kappa_2=0.001\,\mbox{J/m}$;
the healing (used before for analytical reasons) was effectively not considered, 
cf.\ Remark~\ref{rem-no-healing}.

\begin{figure}[th]
\begin{center}
\psfrag{W}{\small $\Omega$}
\psfrag{GN}{\footnotesize $\GN$}
\psfrag{GD}{\footnotesize $\GD$}
\psfrag{GDN}{\footnotesize $\GD/\GN$}
\psfrag{uD}{\footnotesize $\uD\!=\uD^{}(t)$}
\psfrag{5/6}{\footnotesize $\displaystyle{\frac56}$}
\psfrag{1}{\footnotesize $1$\,m}
\includegraphics[width=.9\textwidth]{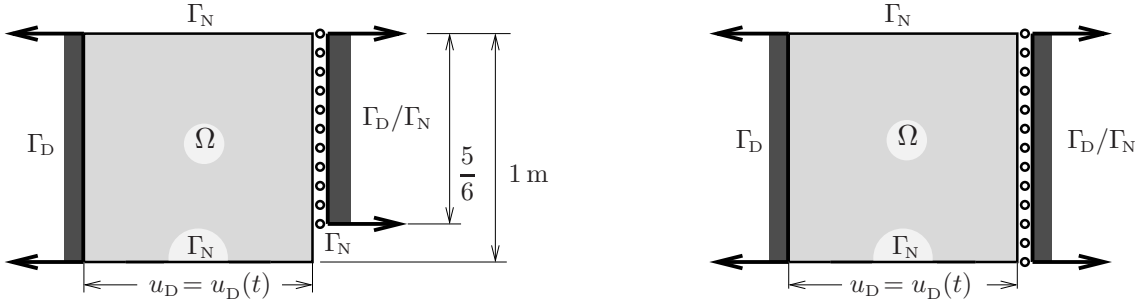}
\end{center}
\vspace*{-1em}
\caption{\sl Geometry of a 2-dimensional square-shaped specimens 
subjected to two tension-loading experiments; the 
right-hand side of the rectangle $\Omega$ combines Dirichlet condition in the horizontal
direction and homogeneous Neumann condition in the vertical direction.}
\label{fig4-dam-plast-geom}
\end{figure}
\medskip\noindent{\sl The specimen and its loadings}:
We consider a 2-dimensional square-shaped specimen subjected to two slightly 
different loading regimes. Both of them consist in a pure ``hard-devise'' horizontal 
load by Dirichlet boundary conditions with the left-hand side $\GD$ fully fixed while the 
right-hand side $\GD/\GN$ combines time-varying Dirichlet condition in the 
horizontal direction with the Neumann condition in the vertical direction.
The only (intentionally small) difference is in keeping a small
bottom part of this vertical side free (see Fig.\,\ref{fig4-dam-plast-geom}-left)
or not (see Fig.\,\ref{fig4-dam-plast-geom}-right).
As our model is fully rate-independent, the time scale is irrelevant
and we thus consider a dimensional-less process time $t\in[0,80]$ 
controlling the linearly growing hard-devise (=\,Dirichlet) load until 
the maximal horizontal shift 80\,mm of the right-hand side $\GD/\GN$.


\medskip\noindent{\sl The discretisation}: In comparison with Section~\ref{sect-BEM}, we 
dare make a shortcut by neglecting the gradient term $\nabla\pi$ in the stored energy 
\eqref{eq5:plast-dam-E} by putting $\kappa_1=0$, which allows for using only P0-elements 
for $\pi$. It also allows for transformation of the cost functional of
\eqref{minimization-1-implem} to a functional of the variable $u$ only by substituting 
the elementwise dependency of $\pi$ on $u$, see \cite{AlCaZa99ANAPEH,CeKoSyVa14TDDSEP} 
for more details. Then, the quasi-Newton iterative method mentioned in Section~\ref{sect-BEM} 
is applied to solve $u_{\tau h}^k$ while $\pi_{\tau h}^k$ is reconstructed from it. More 
details on this specific elasto-plasticity solver can be found e.g.\ in 
\cite{CeKoSyVa14TDDSEP,GKNT10ANAT,GruVal09SOTSPESNM}. 
Here, the spatial P1/P0 FEM discretisation of the rectangular domain $\Omega$ 
uses a uniform triangular mesh with $2304$ 
elements and $1201$ nodes. The code was implemented in Matlab, being
available for download and testing at Matlab Central as a 
package \href{http://www.mathworks.com/matlabcentral/fileexchange/49943-continuum-undergoing-combined-elasto-plasto-damage-transformation}{\textit{Continuum undergoing combined elasto-plasto-damage transformation}}, cf.\ \cite{code}.
It is based on an original elastoplasticity code related to multi-threshold models \cite{BrCaVa05QSBVPMS},
here simplified for a single-threshold case. 
It partially utilizes vectorization techniques of \cite{RaVa13FMAFEM2D3D} 
and works reasonably fast also for finer triangular meshes. In contrast to the fixed 
spatial discretisation, we consider three time discretisation to document the convergence 
(theoretically stated only for unspecified subsequences in 
Proposition~\ref{ch4:prop-dam-plast}) on particular computational experiments. More
specifically, we used three time steps $\tau=1$, $0.1$, or $0.01$, i.e.\ the
equidistant partition of the time interval $[0,80]$ to 80, 800, or 8000 time steps, 
respectively.
%
%

\medskip\noindent{\sl Simulation results}:
The averaged stress/strain (or rather force/stretch) response is depicted 
in Figure~\ref{fig-stress}. Notably, after damage is completed, some stress 
still remains (as is nearly independent on further stretch because the elastic 
moduli $\lambda_0$ and $\mu_0$ are considered very small). These remaining 
stresses are caused by non-uniform plastification of the specimen during 
the previous phases of the loading. 
\begin{figure}
\center
\includegraphics[width=0.39\textwidth]{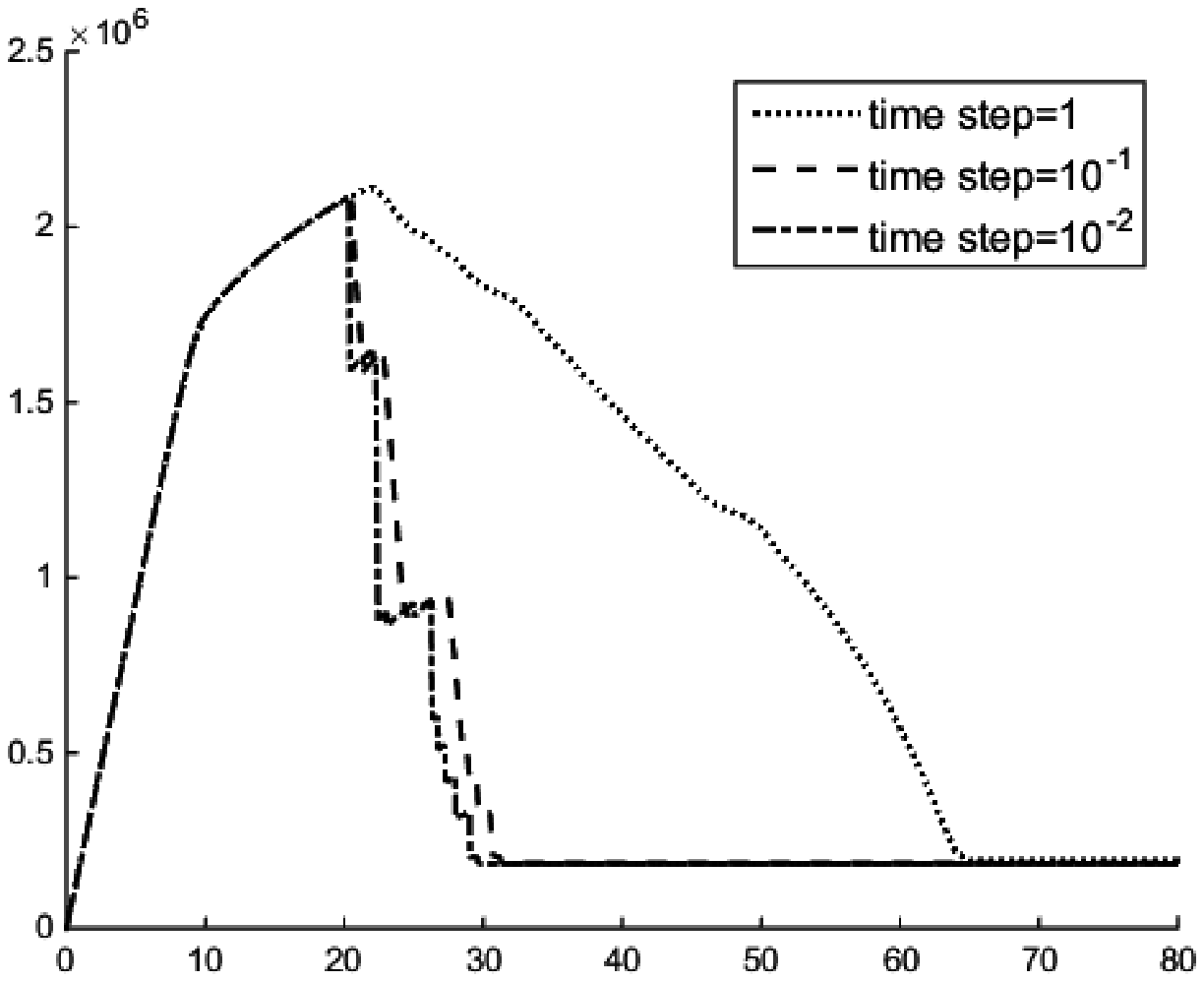}\hspace*{3em}
\includegraphics[width=0.39\textwidth]{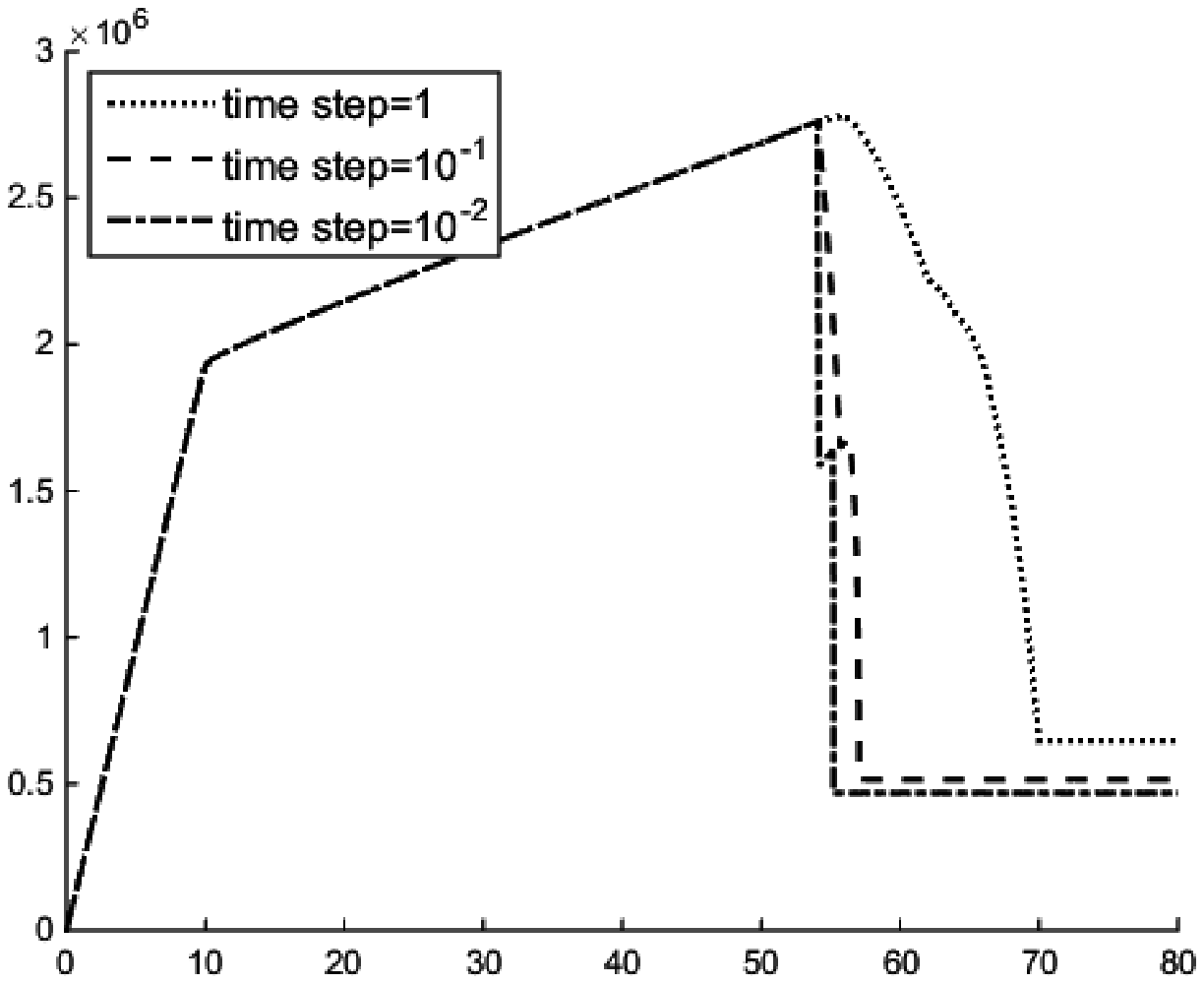}
\caption{\sl Evolution of averaged elastic von Mises stresses $\int_\Omega|\dev\sigma_{\rm el}|\,\d x$
over time for the two experiments from Fig.\,\ref{fig4-dam-plast-geom} for three different
time discretisations. The left 
experiment ruptures earlier under less stretch and leaving less plastification and 
remaining stress comparing to the right experiment. In both experiments, the convergence
theoretically supported by Proposition~\ref{ch4:prop-dam-plast} is well documented.
}
\label{fig-stress}\end{figure}
One can also note that Figure~\ref{fig-stress}-right imitates quite well the 
scenario from Figure~\ref{fig5:plast-dam}-right while Figure~\ref{fig-stress}-left 
is rather a mixture of both regimes from Figure~\ref{fig5:plast-dam} and,
interestingly, the rupture proceeds in three stages. The respective spacial 
distribution of the evolving state variable is depicted at few selected 
instants on Figures~\ref{fig-movie-asym} and \ref{fig-movie-sym}. 
\begin{figure}
\center
{\scriptsize\bf DAMAGE $\zeta$\hspace{2.5em}PLASTIC STRAIN $|\pi|$\hspace{2em}STRESS $|\dev\sigma_{\mathrm{el}}|$\hspace{2.5em}RESIDUUM ${\rm log}\,|R|$\hspace*{3em}}
\\[.5em]\includegraphics[width=0.75\textwidth]{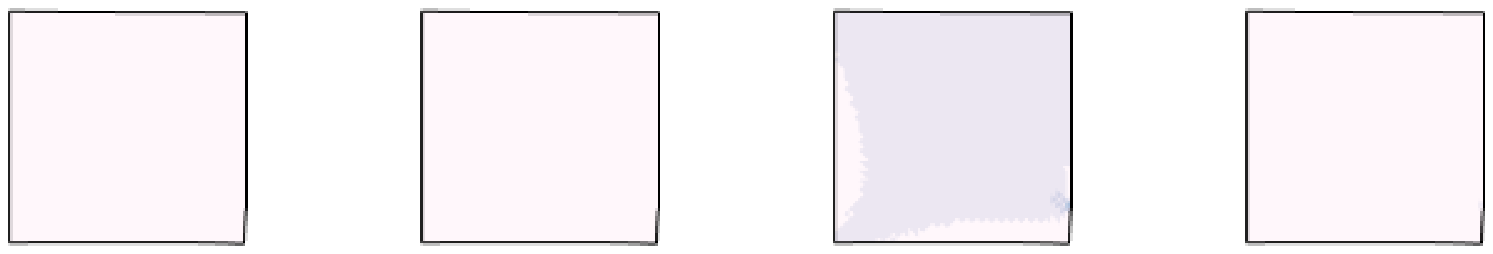}\quad\TIME{4\ \ }
\\[.3em]
\includegraphics[width=0.75\textwidth]{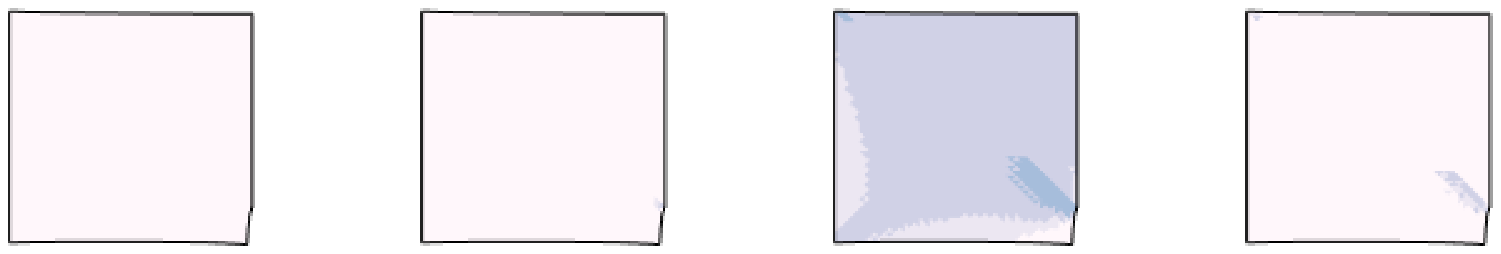}\quad\TIME{8\ \ }
\\[.3em]
\includegraphics[width=0.75\textwidth]{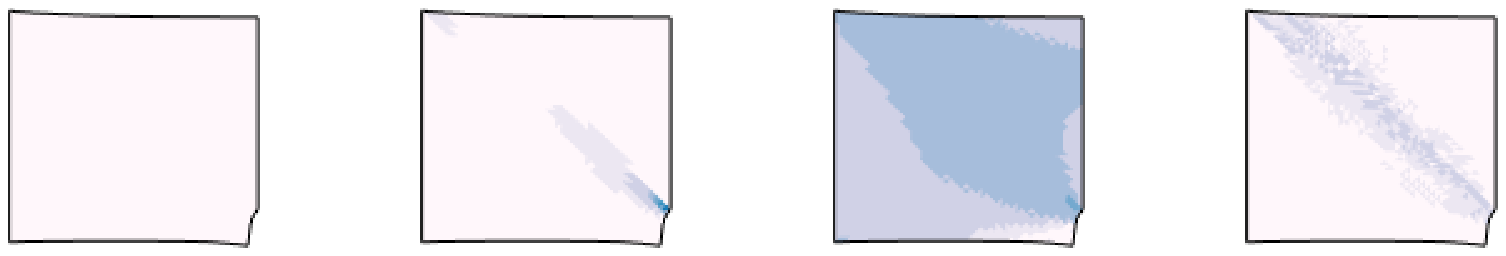}\quad\TIME{12}
\\[.3em]
\includegraphics[width=0.75\textwidth]{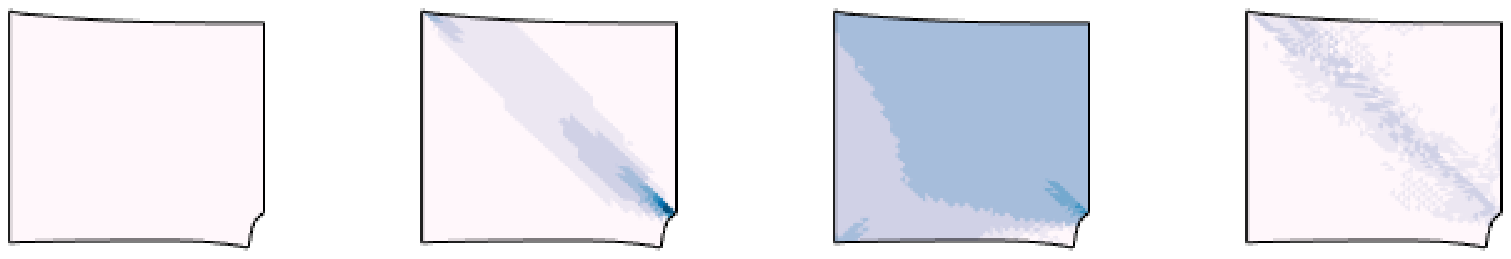}\quad\TIME{16}
\\[.3em]
\includegraphics[width=0.75\textwidth]{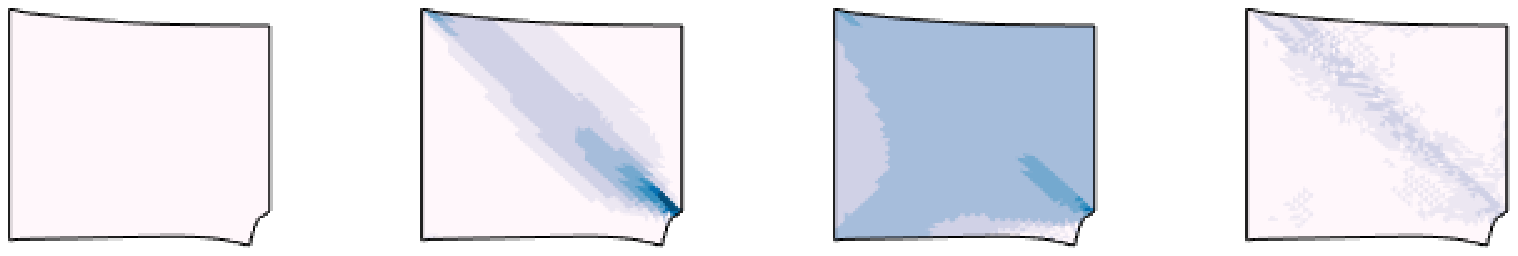}\quad\TIME{20}
\\[.3em]
\includegraphics[width=0.75\textwidth]{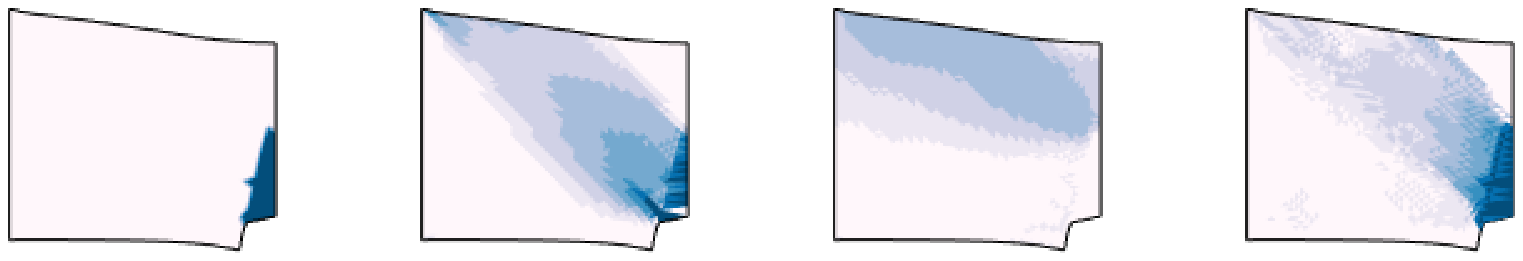}\quad\TIME{24}
\\[.3em]
\includegraphics[width=0.75\textwidth]{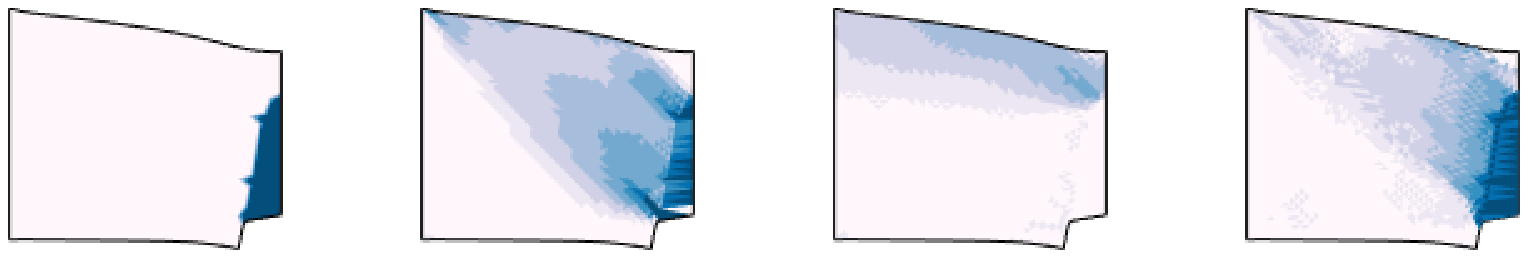}\quad\TIME{28}
\\[.3em]
\includegraphics[width=0.75\textwidth]{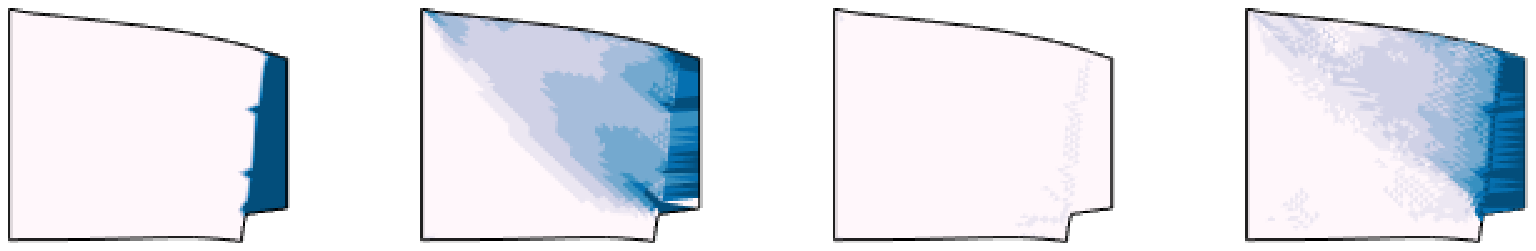}\quad\TIME{32}
\\[.3em]
\hspace{0.02\textwidth}\includegraphics[width=0.73\textwidth]{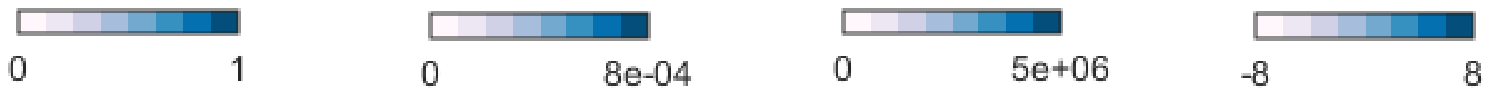}
\hspace{0.1\textwidth}
\caption{\sl Evolution of spatial distribution of the state $(u,\pi,\zeta)$ with also
the von Mises stress and the residuum $R$ from \eqref{rhs-to-evaluate} 
at (equidistantly) selected instants
for the asymmetric geometry from Fig.\,\ref{fig4-dam-plast-geom}-left.
The deformation is visualized by a displacement $u$ magnified 250$\,\times$, and $\tau=0.1$ was used. 
Damage occurs relatively early on the right-hand side due to the stress 
concentration and propagates in several partial steps, 
cf.\ Fig.\,\ref{fig-stress}-left.}
\label{fig-movie-asym}
\end{figure}
\begin{figure}
\center
{\scriptsize\bf DAMAGE $\zeta$\hspace{2.5em}PLASTIC STRAIN $|\pi|$\hspace{2em}STRESS $|\dev\sigma_{\mathrm{el}}|$\hspace{2.5em}RESIDUUM ${\rm log}\,|R|$\hspace*{3em}}
\\[.5em]
\includegraphics[width=0.75\textwidth]{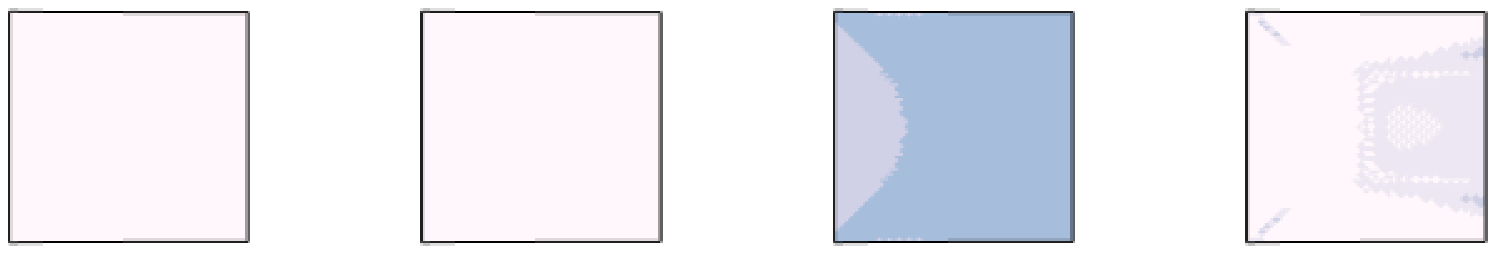}\quad\TIME{10}
\\[.3em]
\includegraphics[width=0.75\textwidth]{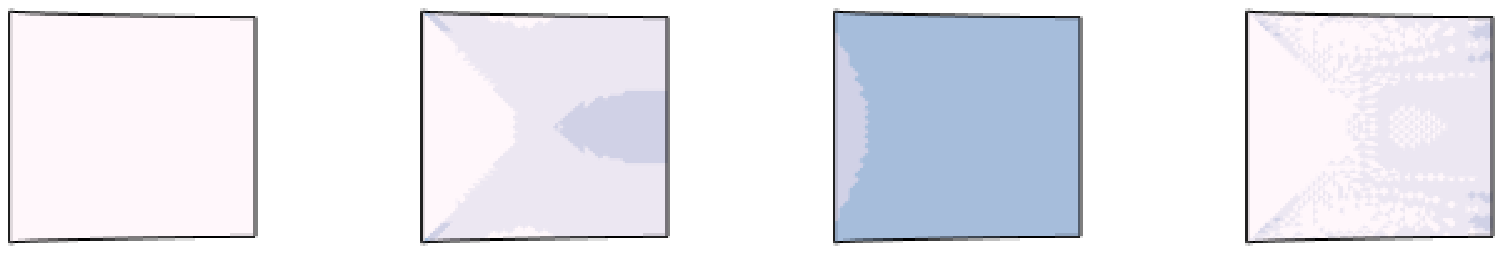}\quad\TIME{20}
\\[.3em]
\includegraphics[width=0.75\textwidth]{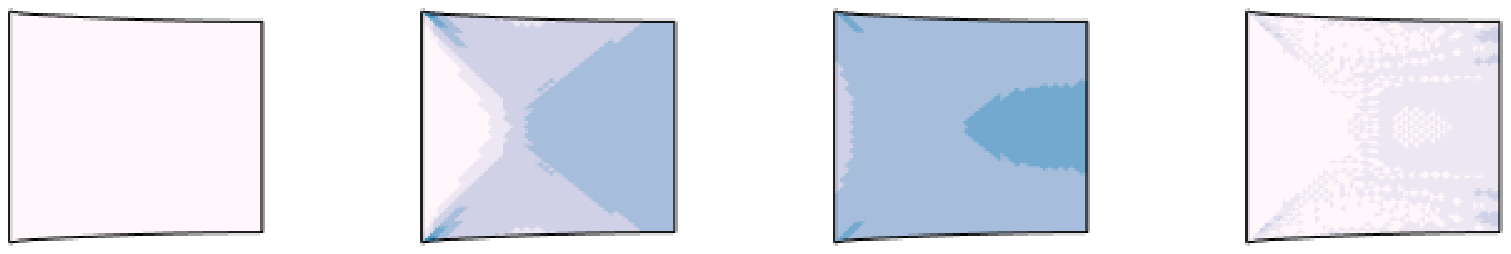}\quad\TIME{30}
\\[.3em]
\includegraphics[width=0.75\textwidth]{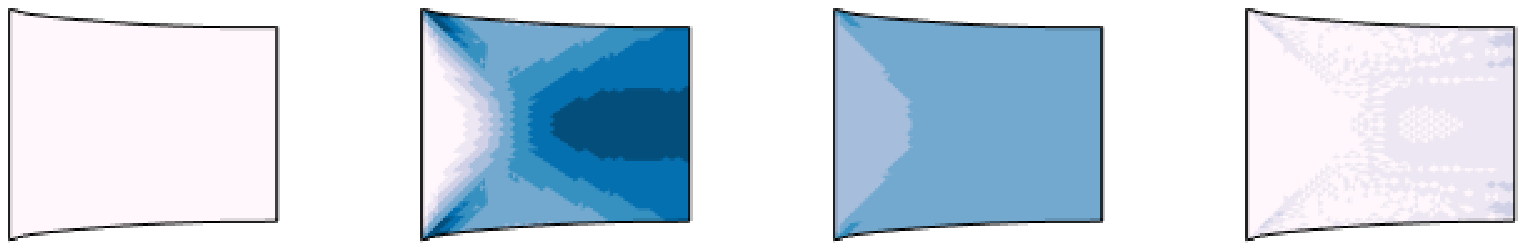}\quad\TIME{50}
\\[.3em]
\includegraphics[width=0.75\textwidth]{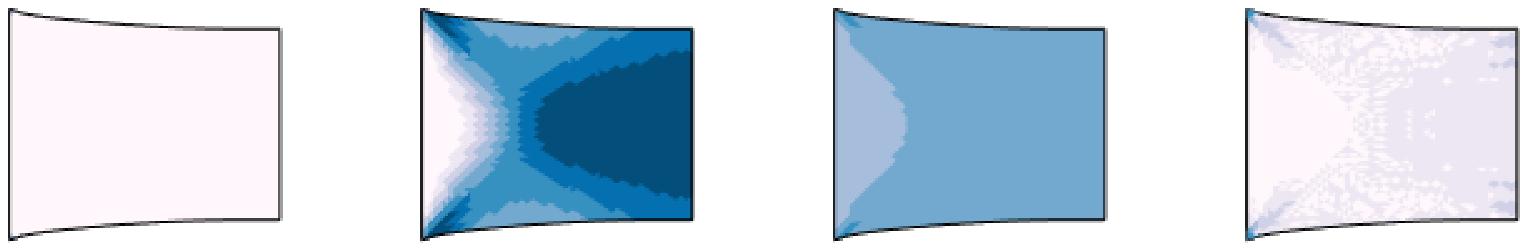}\quad\TIME{54}
\\[.3em]
\includegraphics[width=0.75\textwidth]{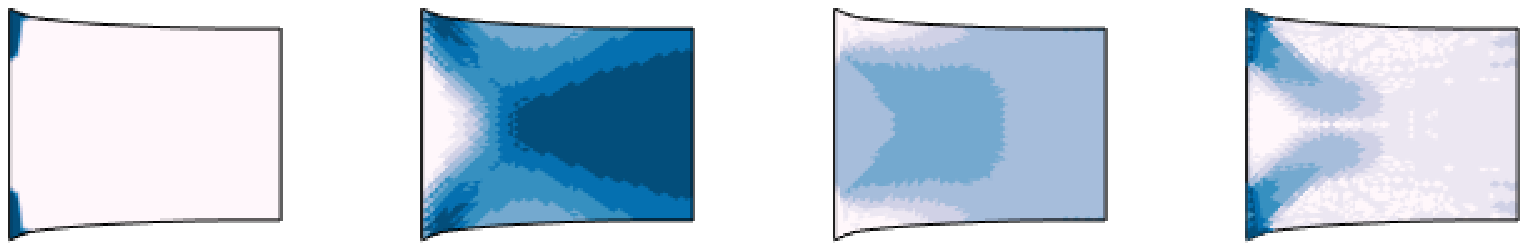}\quad\TIME{55}
\\[.3em]
\includegraphics[width=0.75\textwidth]{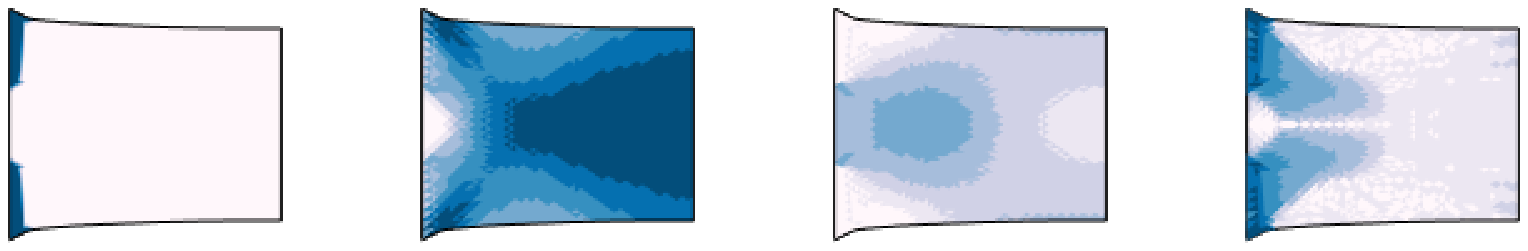}\quad\TIME{56}
\\[.3em]
\includegraphics[width=0.75\textwidth]{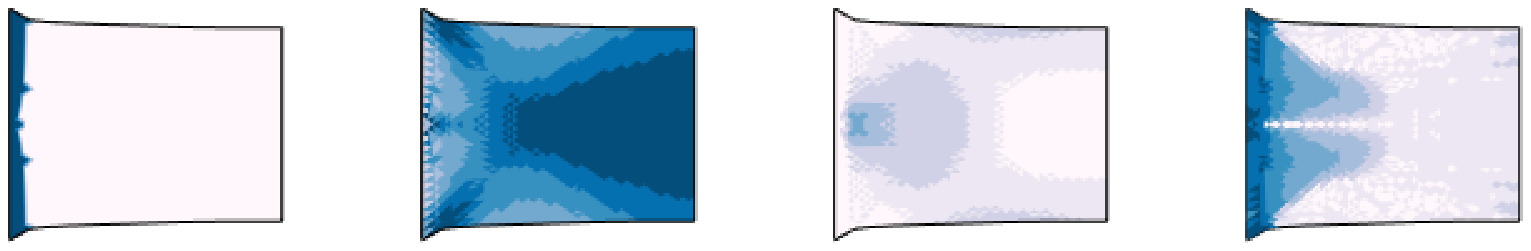}\quad\TIME{57}
\\[.3em]
\includegraphics[width=0.75\textwidth]{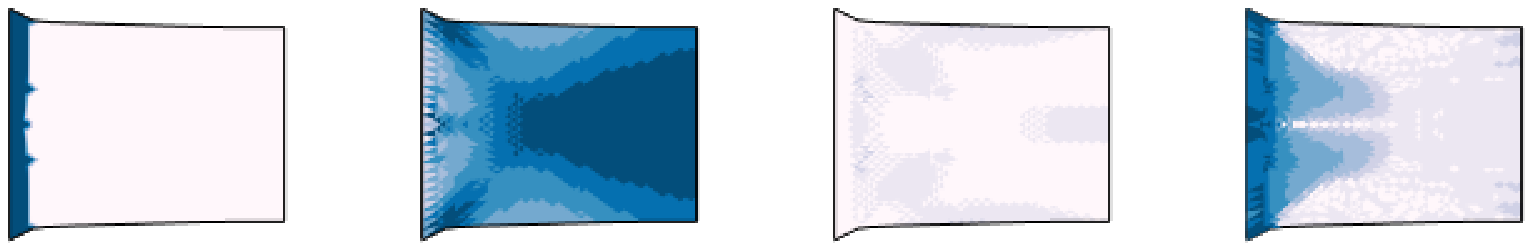}\quad\TIME{60}
\\[.3em]
\hspace{0.02\textwidth}\includegraphics[width=0.73\textwidth]{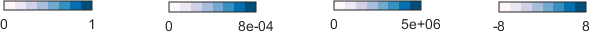}
\hspace{0.1\textwidth}
\caption{\sl Evolution of spatial distribution of the state $(u,\pi,\zeta)$ with also
the von Mises stress and the residuum $R$ from \eqref{rhs-to-evaluate} at selected instants
for the symmetric geometry from Fig.\,\ref{fig4-dam-plast-geom}-right.
The deformation is visualized by a displacement $u$ magnified 250$\,\times$, and $\tau=0.1$ 
was used. The process inherits the symmetry of the specimen and loading. In contrast to 
Fig.\,\ref{fig-movie-asym}, damage occurs rather later on the left-bottom corner and 
propagates fast, hence the snapshots are selected not in an equidistant way here.}
\label{fig-movie-sym}
\end{figure}
It is well seen how a relatively small variation of geometry in Figure~\ref{fig4-dam-plast-geom}
dramatically changes the spatial scenario and triggers damage in very different spots
of the specimen. This is an expected notch-effect causing stress concentration 
and relatively early initiation of cracks at such spots, i.e.\ here such a notch is 
the point of the transition $\GN$ to $\GD/\GN$ in Figure~\ref{fig4-dam-plast-geom}-left.
The AMDP suggested in Remark~\ref{rem-DMP} is depicted in 
Figures~\ref{fig:residua_plus_diss_energy} and \ref{fig:residua}. It should 
be emphasized that the maximum-dissipation principle (as devised originally 
by Hill \cite{Hill48VPMP}) is reliably satisfied only for convex 
stored energies as occurs during mere plastification phase, as also 
seen in Figure~\ref{fig:residua}, while in general it does not need
to be satisfied even in obviously physically relevant stress-driven 
evolutions, as already mentioned in Remark~\ref{rem-IMDP} 
and which can be expected even here during massive fast rupture 
of a wider regions (but in spite of it, Figure~\ref{fig:residua_plus_diss_energy}
shows a good satisfaction of AMDP even during such rupture phases and
in some sense demonstrate a good applicability of the model and solution 
concept and its algorithmic realization).
%
%

\begin{remark}[{\sl Symmetry issue}]\label{rem-symmetry}
\upshape
Actually, one could understood the square 1$\,\times\,1$
in Fig.~\ref{fig4-dam-plast-geom} as one half of a rectangle with 
sides 2$\,\times\,1$ with the right-hand side of the 
 1$\,\times\,1$ square being the symmetry axis of the 2$\,\times\,1$
rectangle which is then loaded from the vertical
sides fully symmetrically. Engineers actually most routinely assume that such 
symmetry of this geometry would be inheritted by all (or at least by one) 
solution(s) and use the reduced geometries on Fig.~\ref{fig4-dam-plast-geom} 
for calculations of the full  2$\,\times\,1$-rectangle. We intentionally did 
not use this interpretation because, in fact, one can only
say that the set of all solutions inherits the (possible) symmetry of
the specimen and its loading but not particular solutions, and even 
it may be that there is no solution inheritting this symmetry or 
that experimental evidence shows preferences for nonsymmetric solutions. 
Cf. the discussion in \cite{GaMaGr15DFMI,PanMan13SNDF}. In addition,
the geometry in Fig.~\ref{fig4-dam-plast-geom}-left would
lead to a 2$\,\times\,1$ rectangle with a partial ``cut'' in the mid-bottom
side, which is not a Lipschitz domain.
\end{remark}

\begin{figure}
\center
\includegraphics[width=0.39\textwidth]{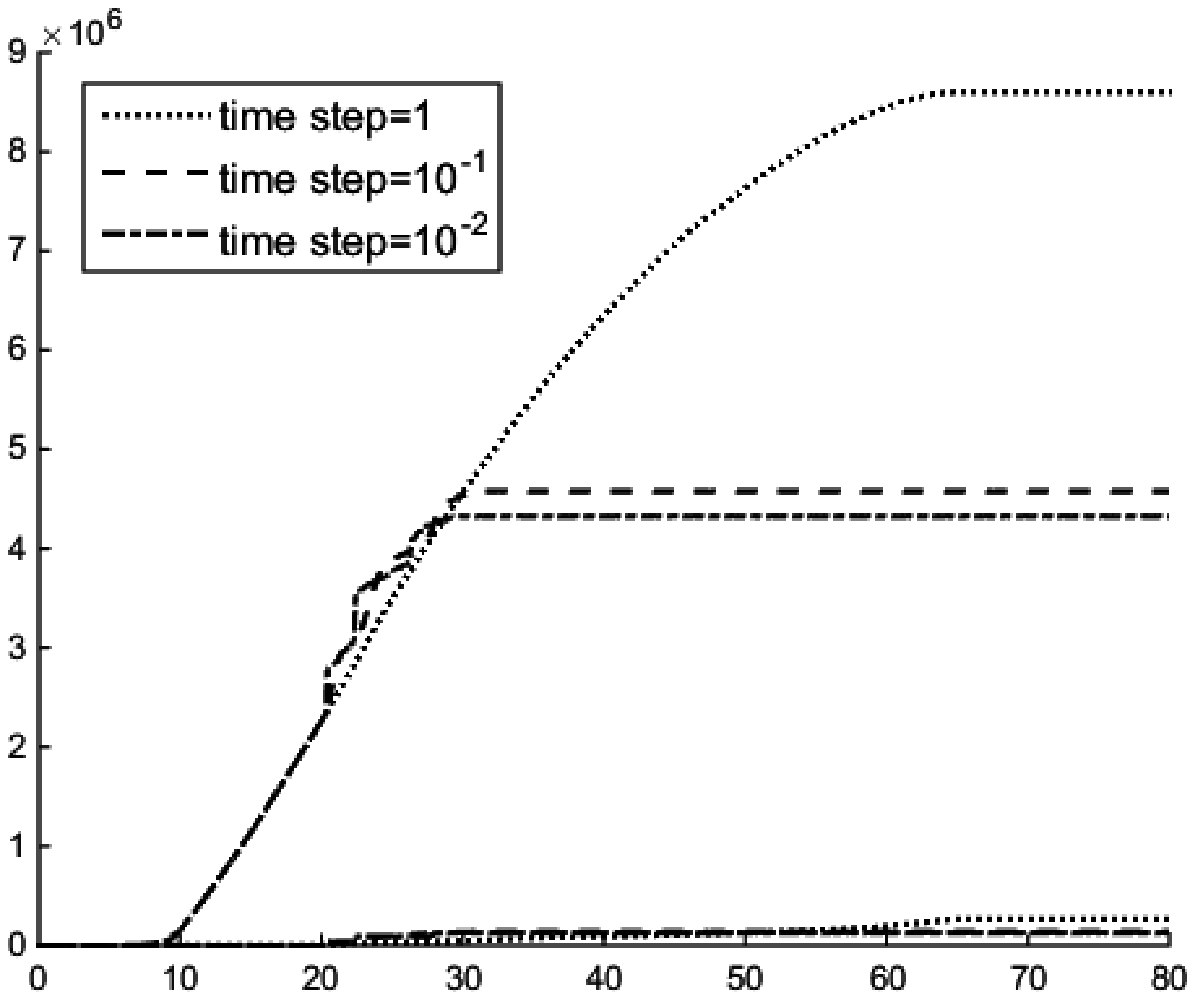}\hspace*{3em}
\includegraphics[width=0.39\textwidth]{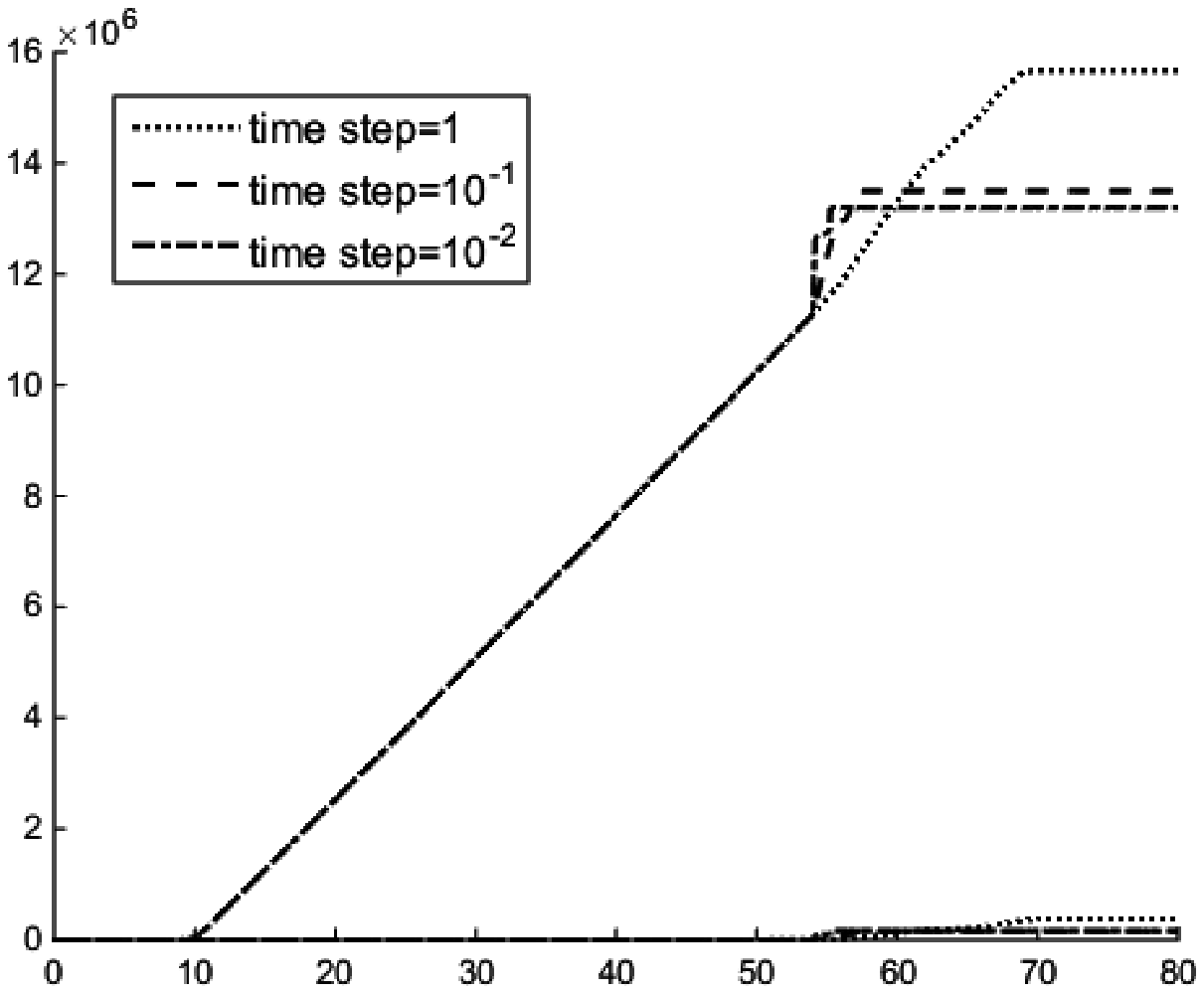}
\caption{\sl Evolution of the dissipated energy $\Diss_{\scrR}^{}(z;[0,t])$ 
(top 3 curves) and integrated residua $\int_0^t\int_\Omega R\,\d x\d t$ (bottom 3 curves) 
in two experiments from Figure\,\ref{fig4-dam-plast-geom}.
In particular, it again shows good tendency of convergence and, moreover, that 
the violation of the approximate maximum-dissipation principle is small with 
respect to the overall dissipated energy and the evolution was stress driven
with a good accuracy about 1-2\%.
}
\label{fig:residua_plus_diss_energy}\end{figure}

\begin{figure}
\center
\includegraphics[width=0.39\textwidth]{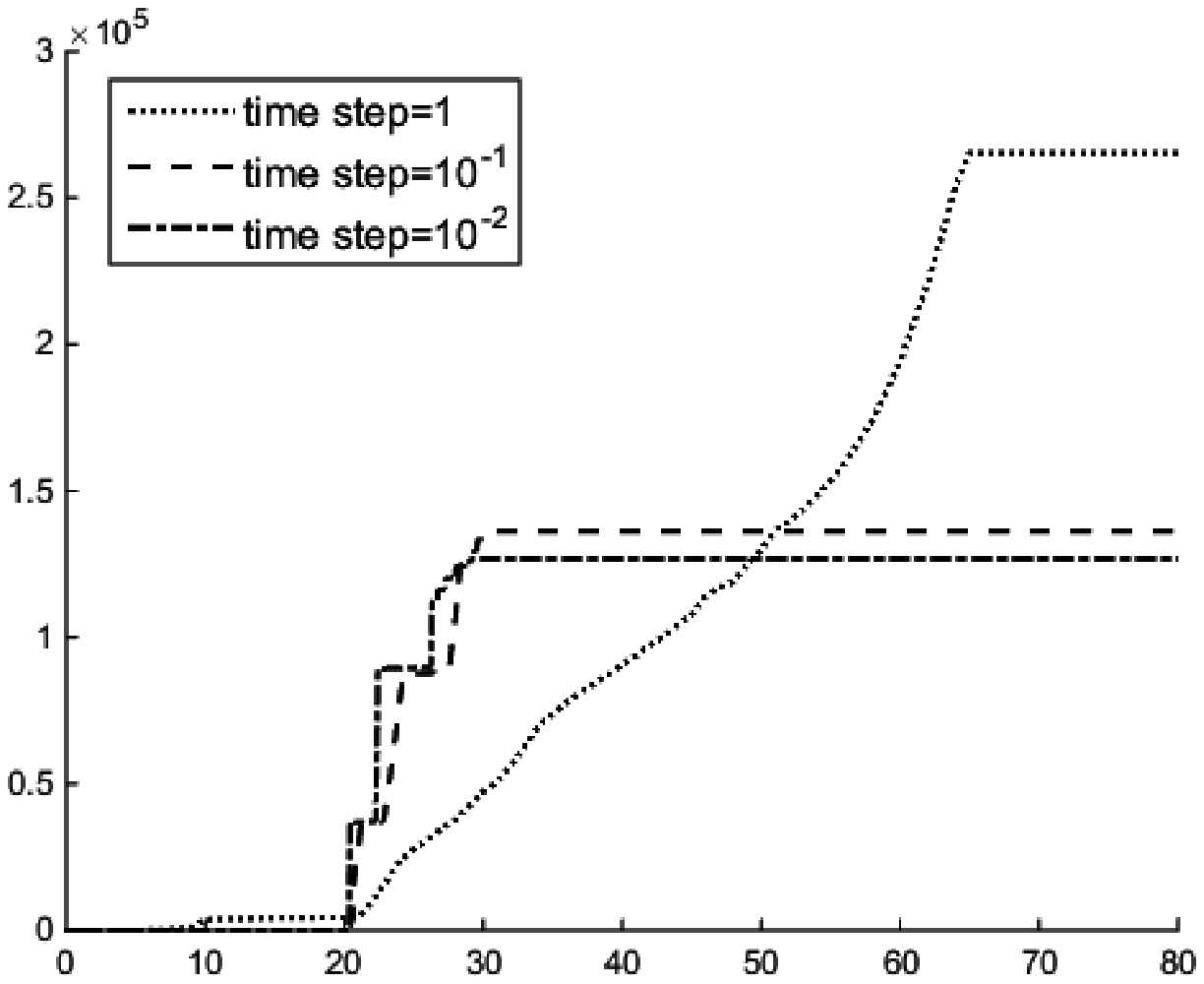}\hspace*{3em}
\includegraphics[width=0.39\textwidth]{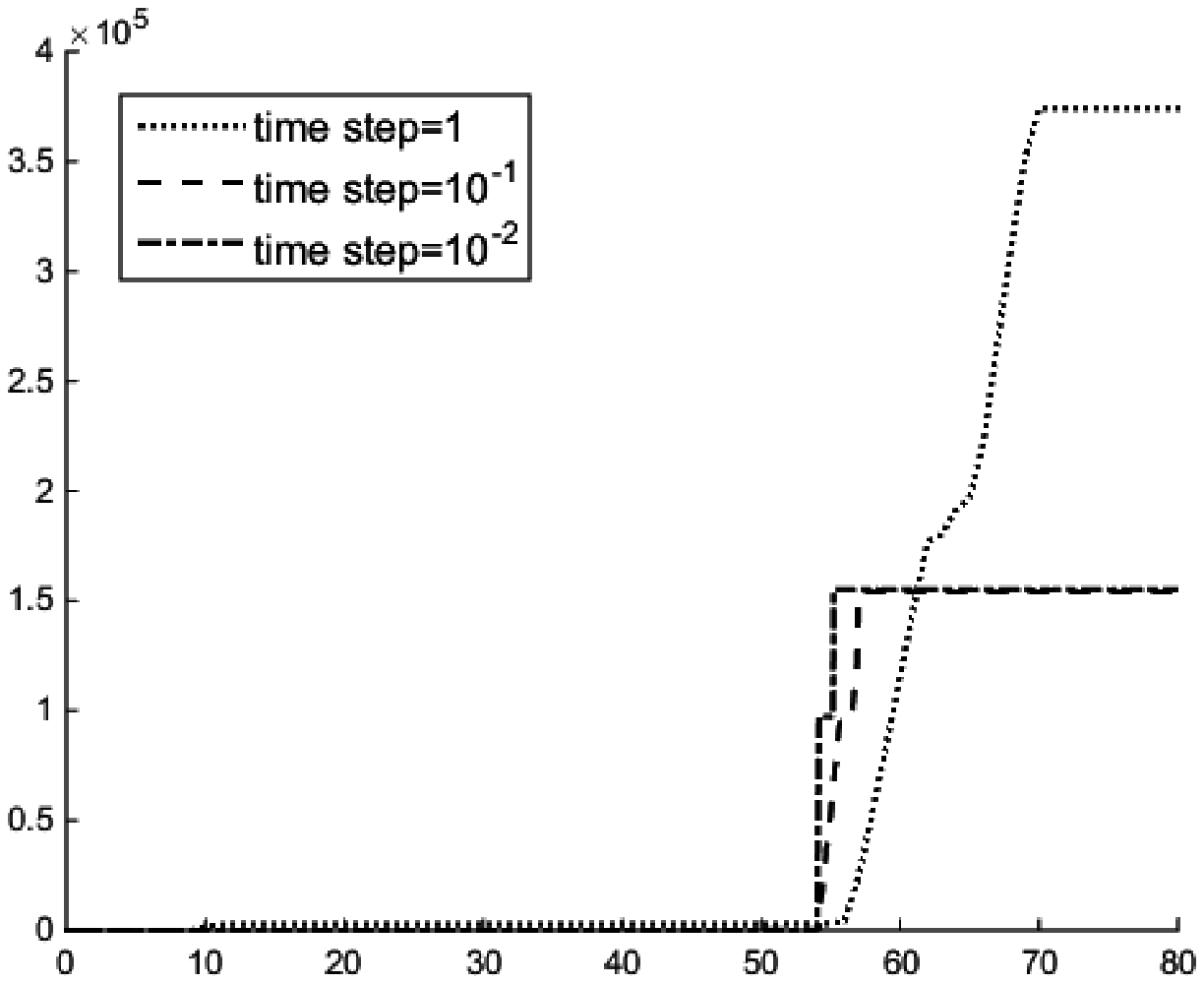}
\caption{\sl A detailed scaling of the bottom 3 curves (=\,the residua in AMDP) from 
Figure \ref{fig:residua_plus_diss_energy}.
A good convergence to zero is seen in plastification period while some residuum 
is generated during damage period where nonconvexity of the stored energy truly 
comes into effect. 
}
\label{fig:residua}\end{figure}

\begin{remark}[{\sl Recovery of the integrated maximum-dissipation principle IMDP}]\label{rem-MDP-recovery}
\upshape
It should be emphasized that, even if the intuitively straightforward AMDP is 
asymptotically satisfied, the recovery of even the less-selective IMDP 
\eqref{IMDP+} for $\tau\to0$ is not clear. This is obviously related with 
instability of IMDP under data perturbation if $\scrE(t,\cdot)$ is not convex. 
Here, to recover the IMDP on $I$,  it would suffice to show that for 
all $\eps>0$ there is $\tau_\eps>0$ such that
for any $0<\tau\le\tau_\eps$
it holds 
\begin{align}\label{stronger-AMDP-}
\sum_{k=1}^{T/\tau}\inf_{t\in[k\tau-\tau,k\tau]
}\big\langle\xi(t),z(k\tau){-}z(k\tau{-}\tau)\big\rangle
-\scrR\big(z(k\tau){-}z(k\tau{-}\tau)\big)\ge-\eps
\end{align}
for some selection $\xi(t)\in-\partial_z\scrE(t,u(t),z(t))$, cf.\ the definition
\eqref{def-of-RS-int} and realize that  the equi-distant partitions 
are cofinal in all partitions of $I$. This can be guaranteed only under rather strong conditions,
namely if, for all $\eps>0$, there is $\tau_\eps>0$ such that
for any $0<\tau\le\tau_\eps$,
the following strengthened version of the AMDP 
\begin{align}\label{stronger-AMDP}
\sum_{k=1}^{T/\tau_\eps}
\scrR\big(\barz_\tau(t_k)-\barz_\tau(t_{k-1})\big)-\big\langle\barxi_\tau(t),
\barz_\tau(t_k)-\barz_\tau(t_{k-1})\big\rangle\le\eps
\end{align}
holds for $t_k=k\tau_\eps$, any $t_{k-1}\le t\le t_k$, and some 
$\barxi_\tau(t)\!\in\!-\partial_z\scrE(t,\baru_\tau(t),\barz_\tau(t))$, 
and if 
$\barxi_\tau(t)\weak\xi(t)
$, which can be assumed due to the available a-priori estimates
used in the proof of Pro\-po\-sition~\ref{ch4:prop-dam-plast}.
Using also \eqref{dam-plast-conv} and the
(norm,weak)-upper semicontinuity of 
 $\partial_z\scrE(t,\cdot,\cdot)$, in the limit for $\tau\to0$, 
from such a strengthened AMDP, one can read 
$\sum_{k=1}^{T/\tau_\eps}
\scrR(z(t_k)-z(t_{k-1}))-\langle\xi(t),z(t_k)-z(t_{k-1})\rangle\le\eps$
for any $t_{k-1}\le t\le t_k$ and for some $\xi(t)\in-\partial_z\scrE(t,u(t),z(t))$,
from which \eqref{stronger-AMDP-} indeed follows. In fact, our intuitive version 
of AMDP from Remark~\ref{rem-DMP} computationally verified \eqref{stronger-AMDP}
in Figure~\ref{fig:residua} in particular examples for  
$\tau=\tau_\eps$ only.
And, on top of it, we would need \eq{stronger-AMDP} to be shown rather for
$\bar\xi_\tau
\in(-\partial_\pi\scrE(t,\baru_\tau(t),\barz_\tau(t{-}\tau)),
-\partial_\zeta\scrE(t,\baru_\tau(t),\barz_\tau(t)))$.
\end{remark}

{\small\it\noindent
Acknowledgments. This research has been supported by GA\,\v CR through the projects
13-18652S ``Computational modeling of damage and transport processes 
in quasi-brittle materials'' and 14-15264S 
``Experimentally justified multiscale modelling of shape memory alloys''
with also the also institutional support RVO:61388998 (\v CR).
}

\baselineskip=13pt

\bibliographystyle{plain}
\bibliography{tr-jv-plast-dam-MDP}

\end{document}